\documentclass[11pt]{article}
\usepackage[T1]{fontenc}
\usepackage{amsfonts}
\usepackage{amsmath}
\usepackage{amssymb}
\usepackage{amsthm}
\usepackage{bbm}
\usepackage{bm}
\usepackage{mathrsfs}
\usepackage{verbatim}
\usepackage{setspace}
\usepackage{color}
\usepackage{pdfsync}
\usepackage{enumitem}
\usepackage{graphicx}
\usepackage{subfigure}
\usepackage{tikz}
\usetikzlibrary{automata, positioning}

\theoremstyle{plain}
\newtheorem{theorem}{Theorem}[section]
\newtheorem{proposition}[theorem]{Proposition}
\newtheorem{lemma}[theorem]{Lemma}
\newtheorem{corollary}[theorem]{Corollary}

\theoremstyle{definition}
\newtheorem{definition}[theorem]{Definition}
\newtheorem{remark}[theorem]{Remark}

\newtheorem{example}[theorem]{Example}

\newtheorem{assumption}[theorem]{Assumption}

\theoremstyle{remark}

\renewenvironment{thebibliography}[1]{%
\begin{oldthebibliography}{#1}%
\setlength{\baselineskip}{1em}
\linespread{.3}
\small
\setlength{\parskip}{0.9ex}%
\setlength{\itemsep}{.3em}%
}%
{%
\end{oldthebibliography}%
}
\newcommand{\eps}{\varepsilon}

\newcommand{\N}{\mathbb{N}}

\newcommand{\R}{\mathbb{R}}

\newcommand{\X}{\mathsf{X}}
\newcommand{\Y}{\mathsf{Y}}

\newcommand{\cL}{\mathcal{L}}

\newcommand{\fT}{\mathfrak{T}}

\DeclareMathOperator{\dom}{dom}

\DeclareMathOperator{\graph}{graph}
\DeclareMathOperator{\proj}{proj}

\DeclareMathOperator{\spt}{spt}
\DeclareMathOperator{\Unif}{Unif}
\DeclareMathOperator{\Int}{int}

\DeclareMathOperator{\id}{Id}

\DeclareMathOperator*{\argmin}{arg\, min}

\DeclareMathOperator{\conv}{conv}

\newcommand{\1}{\mathbf{1}}

\newcommand{\br}[1]{\langle #1 \rangle}

\newcommand{\qmbox}[1]{\quad\mbox{#1}\quad}
\newcommand{\qforallq}{\quad\mbox{for all}\quad}
\newcommand{\qforq}{\quad\mbox{for}\quad}
\newcommand{\qwhereq}{\quad\mbox{for}\quad}

\newcommand{\tsum}{{\textstyle \sum}}

\newcommand{\mykill}[1]{}
\numberwithin{equation}{section}
\usepackage[pdfborder={0 0 0}]{hyperref}
\hypersetup{
  urlcolor = black,
  pdfauthor = {Espen Bernton, Promit Ghosal, Marcel Nutz},
  pdfkeywords = {Optimal Transport; Entropic Penalization; Large Deviations; Cyclical Invariance},
  pdftitle = {Entropic Optimal Transport: Geometry and Large Deviations},
  pdfsubject = {Entropic Optimal Transport: Geometry and Large Deviations},
  pdfpagemode = UseNone
}
\begin{document}

\title{\vspace{-1em}
 Entropic Optimal Transport:\\Geometry and Large Deviations%
 \thanks{
 The authors are grateful to Julio Backhoff-Veraguas, Guillaume Carlier, Wilfried Gangbo and Jon Niles-Weed for insightful discussions that greatly contributed to this research.}
 }%
\date{\today}
\author{
  Espen Bernton%
  \thanks{
  Department of Statistics, Columbia University, eb3311@columbia.edu.
  }
  \and
  Promit Ghosal%
  \thanks{Department of Mathematics, Massachusetts Institute of Technology, promit@mit.edu.
  }
  \and
  Marcel Nutz%
  \thanks{
  Departments of Statistics and Mathematics, Columbia University, mnutz@columbia.edu. Research supported by an Alfred P.\ Sloan Fellowship and NSF Grants DMS-1812661, DMS-2106056.}
  }
\maketitle \vspace{-1.2em}

\begin{abstract}
  We study the convergence of entropically regularized optimal transport to optimal transport. The main result is concerned with the convergence of the associated optimizers and takes the form of a large deviations principle quantifying the local exponential convergence rate as the regularization parameter vanishes. The exact rate function is determined in a general setting and linked to the Kantorovich potential of optimal transport. Our arguments are based on the geometry of the optimizers and inspired by the use of $c$-cyclical monotonicity in classical transport theory. The results can also be phrased in terms of Schr\"odinger bridges.
\end{abstract}

\vspace{1em}

{\small

}
\vspace{1em}

\section{Introduction}\label{se:intro}

Over the last three decades, optimal transport theory has flourished due to its connections with geometry, analysis, probability theory, and other fields in mathematics; see for instance \cite{RachevRuschendorf.98a, RachevRuschendorf.98b, Villani.09}. Following computational advances which have enabled high-dimensional applications, a renewed interest comes from applied fields such as machine learning, image processing and statistics. Popularized in this area by Cuturi~\cite{Cuturi.13}, entropic regularization is a key computational approach for high-dimensional problems. The resulting \emph{entropic optimal transport} problem provides an approximate optimal transport when solved for small regularization parameter~$\eps>0$ while admitting much more efficient  algorithms than the unregularized problem, in addition to having other desirable properties. We defer the discussion of related literature to Section~\ref{se:literature} below and proceed with a synopsis of the present study.

Given a continuous cost function $c: \X\times\Y\to\R_{+}$ on Polish probability spaces $(\X,\mu)$ and $(\Y,\nu)$, we consider the entropic optimal transport problem
\begin{equation}\label{eq:EOTintro}
  \inf_{\pi\in\Pi(\mu,\nu)} \int_{\X\times\Y} c \,d\pi + \eps H(\pi|\mu\otimes\nu)
\end{equation}
where $\Pi(\mu,\nu)$ is the set of couplings and $H(\cdot|\mu\otimes\nu)$ denotes relative entropy (or Kullback--Leibler divergence) with respect to the product of the marginals; see Section~\ref{se:cyclInvar} for the formal definitions. The constant $\eps>0$ acts as a regularization parameter; $\eps=0$ recovers the (unregularized) optimal transport problem.
Under mild conditions detailed in Sections~\ref{se:cyclInvar} and~\ref{se:clusterPoints}, respectively, the entropic optimal transport problem admits a unique solution $\pi_{\eps}\in\Pi(\mu,\nu)$ and $\pi_{\eps}$ converges weakly to a solution~$\pi_{*}$ of the unregularized problem. Our main interest is to quantify the \emph{speed} of this convergence $\pi_{\eps}\to\pi_{*}$.

For finite-dimensional linear programs---including optimal transport problems with marginals supported by finite sets---the solution of the entropic regularization is known to converge exponentially fast to a solution of the original problem (in total variation, say). In transport problems with continuous marginals, the situation is quite different even in the most regular examples. For Gaussian marginals on~$\R$ and quadratic costs $c(x,y)=|x-y|^{2}$, direct computation shows that~$\pi_{\eps}$ is Gaussian and $\pi_{*}$ is given by a linear transport (Monge) map~$T$. One finds that the transport cost converges only linearly, $\int c \,d\pi_{\eps}-\int c \,d\pi_{*}=\eps/2+o(\eps)$. The culprit for this slowdown is easily spotted by inspecting the closed-form solution:
the leading term in the cost difference stems from the mass $\pi_{\eps}$ places at a distance of approximately $\sqrt{\eps}$ to the support~$\Gamma$ of~$\pi_{*}$ (that is, the graph of~$T$). See Section~\ref{se:literature} for further discussion on the asymptotics of transport costs and value functions, which have been the main focus of the extant literature on the convergence as $\eps\to0$.

In the present study, we adopt a different, more \emph{local} perspective, from which the Gaussian example is actually encouraging: the density of $\pi_{\eps}$ decays \emph{exponentially} away from~$\Gamma$. Indeed, it is proportional to $e^{-\alpha|y-T(x)|^{2}/\eps}$, where $\alpha>0$ is the quotient of the marginal variances. 

The main result of this paper is a comparable statement in a remarkably general setting; it takes the form of a large deviations principle. We define a function~$I(x,y)$ through the following optimization. In addition to the given point $(x,y)=:(x_{1},y_{1})$, choose finitely many points $(x_{2},y_{2}),\dots, (x_{k},y_{k})$ from the support~$\Gamma$ of the limiting optimal transport $\pi_{*}$, as well as a permutation $\sigma\in\Sigma(k)$. Then, consider the difference
\begin{equation}\label{eq:defIdifference}
  \sum_{i=1}^{k} c(x_{i},y_{i}) - \sum_{i=1}^{k}c(x_{i},y_{\sigma(i)})
\end{equation}
between the pointwise transport costs from~$x_{i}$ to~$y_{i}$ with the costs for the permuted destinations~$y_{\sigma(i)}$. The optimization is to maximize this difference, and we define $I(x,y)$ as the supremum value of~\eqref{eq:defIdifference} over all choices of points and permutations. For $(x,y)\in\Gamma$, the optimality of $\pi_{*}$ implies that~$I(x,y)=0$, because $\Gamma$ is $c$-cyclically monotone. But outside~$\Gamma$, we may typically expect that~$I(x,y)>0$. Part~(a) of our theorem below, the large deviations upper bound, shows that~$I$ is a lower bound for the rate function in the general Polish setting. The matching bound~(b) necessitates a condition on the optimal transport problem that is being approximated---but still holds for the majority of continuous or semi-discrete transport problems of interest. We mainly discuss the uniqueness of Kantorovich potentials (Assumption~\ref{as:uniquePotential}) as a sufficient condition; it also gives rise to an insightful representation of~$I$ as $I(x,y)=c(x,y)- \psi^{c}(y)+\psi(x)$, the difference between the cost~$c(x,y)$ and the solution of the dual optimal transport problem (see Proposition~\ref{pr:dualIndep}).  An alternative condition imposing regularity of the optimal transport (Assumption~\ref{as:crossContinuity}) is also considered. Tacitly assuming the existence of $\pi_{\eps}$ and its weak limit (cf.\ Sections~\ref{se:cyclInvar} and~\ref{se:clusterPoints}), the main result reads as follows.

\begin{theorem}\label{th:mainIntro}
  Let $\Gamma=\spt\pi_{*}$ where $\pi_{*}=\lim_{\eps\to0} \pi_{\eps}$ is the limiting optimal transport
  and define $I: \X\times\Y\to[0,\infty]$ by~\eqref{eq:defI}.
  \begin{itemize}
  \item[(a)] For any compact set $C\subset\X\times\Y$,
  $$
    \limsup_{\eps\to0} \eps \log \pi_{\eps}(C) \leq -\inf_{(x,y)\in C} I(x,y).
  $$ 
  \item[(b)] Let Assumption~\ref{as:uniquePotential} or Assumption~\ref{as:crossContinuity} hold, and consider the sets $\X_{0}=\proj_{\X}\Gamma$ and $\Y_{0}=\proj_{\Y}\Gamma$ of full marginal measure. For any open set $U\subset \X_{0}\times\Y_{0}$,
  $$
    \liminf_{\eps\to0} \eps \log \pi_{\eps}(U) \geq -\inf_{(x,y)\in U} I(x,y).
  $$ 
  \end{itemize} 
\end{theorem} 

The theorem shows in particular that the rate depends (only) on the geometry of $\pi_{*}$, which does not seem to be clear a priori. We mention that our result can also be stated in terms of (static) Schr\"odinger bridges. In this context, it is a large deviations principle for the small-noise (or small-time) limit; cf.\ Section~\ref{se:literature}.

For finitely supported marginals, the density of $\pi_{\eps}$ converges exponentially for any cost function; that is, the rate function is strictly positive outside~$\Gamma$. We shall see that the analogue may fail in the continuous case. Rather, positivity depends on the geometry of the cost. The twist condition (injectivity of $\nabla_{x}c(x,\cdot)$) plays an important role, like in many results on optimal transport. We include affirmative positivity results in particular for quadratic costs, which is the most important case for applications. While not pursued in the present paper, our results should also be useful to derive detailed quantitative bounds on the rate in more specific settings. We may also hope to gain insights into how the rate depends on the dimension.

Geometry is a cornerstone in the now-classical theory of optimal transport, where optimality is captured geometrically by the $c$-cyclical monotonicity of a transport's support. Defined by comparing costs at finitely many points, it yields a powerful tool to derive fundamental results such as stability of optimal transports under weak limits or existence of dual potentials. 
We are not aware of a comparable technique in the literature on entropic optimal transport (or on Schr\"odinger bridges).
In this paper, we exploit a cyclical invariance property satisfied by the density of~$\pi_{\eps}$. The invariance itself can be understood as a reformulation of a classical characterization for~$\pi_{\eps}$ through the solution of the dual problem, the Schr\"odinger potentials. The novelty here lies in exploiting the geometric aspect and working on the primal side, following the spirit of $c$-cyclical monotonicity. Like in classical optimal transport, the arguments are remarkably simple and general once the correct notions are in place. Our technique is a departure from the control-theoretic methods in the related literature.
Case in point, the geometric proof that a weak limit $\pi=\lim_{\eps\to0}\pi_{\eps}$ is an optimal transport (cf.~Proposition~\ref{pr:limitsCyclMonotone}), is nearly trivial compared to the Gamma-convergence technique, even in the general Polish context. (Of course, Gamma-convergence is applicable to many other problems where our technique has no analogue.)

We also emphasize another benefit which may illustrate that cyclical invariance is in fact more than just a reformulation of control theory or  convex analysis: the geometry singles out a unique coupling~$\pi_{\eps}$ even if the value function~\eqref{eq:EOTintro} is infinite and hence the usual notion of solution as a minimizer is not meaningful. This is crucial for instance if costs are quadratic but one of the marginal distributions does not have a finite second moment. Our arguments for the large deviations result apply in that setting without any added difficulty, paralleling the geometric insights in classical optimal transport. (On the other hand, the \emph{existence} of~$\pi_{\eps}$ in the case of infinite value functions is not immediate. We establish it in~\cite{GhosalNutzBernton.21b}, together with a stability theorem for~$\pi_{\eps}$, using the same geometric standpoint.) Indeed, we expect the technique to be useful in several other aspects of entropic optimal transport and Schr\"odinger bridges, and thus the technique may be as important a contribution as the main theorem.

The present paper is organized as follows. After reviewing motivations for our research and related literature in the remainder of this Introduction, Section~\ref{se:cyclInvar} details the basic definitions and introduces cyclical invariance. In Section~\ref{se:clusterPoints}, this notion is used to prove that cluster points of~$\pi_{\eps}$ as $\eps\to0$ have $c$-cyclically monotone support, hence are optimal transports. The main result on large deviations is obtained in Section~\ref{se:largeDeviations}: part~(a) of Theorem~\ref{th:mainIntro} is stated as Corollary~\ref{co:limsup} whereas~(b) is split into  Corollaries~\ref{co:liminf} and~\ref{co:liminf2}, each covering one of the two alternative assumptions. %
Section~\ref{se:positivity} gives examples of settings where the rate function~$I$ is strictly positive outside the support~$\Gamma$, with a focus on quadratic costs. Appendix~\ref{se:factorization} contains  facts about Schr\"odinger bridges and a derivation of the cyclical invariance property. In Appendix~\ref{se:uniquenessOfPotentials}, we detail two general settings where Assumption~\ref{as:uniquePotential} on the uniqueness of Kantorovich potentials is satisfied. Finally, Appendix~\ref{se:Loeper} shows how to translate the results on the positivity of~$I$ in Section~\ref{se:positivity} from quadratic costs to more general cost functions by means of $c$-convex analysis.

\subsection{Related Literature}\label{se:literature} 

In the literature on finite-dimensional linear programs and their entropic regularization, the early work~\cite{CominettiSanMartin.94} contains a very detailed study of primal and dual convergence, expansion of the value function, and characterizations of the rates. Their setting includes discrete optimal transport problems with marginals supported by finitely many points, and in that case the pointwise results in~\cite{CominettiSanMartin.94} certainly include the large deviations result for~$\eps\to0$. On the other hand, our main theorem is most relevant when at least one marginal support is connected, hence is complementary to the discrete case. More recently, \cite{Weed.18} proved an exponential convergence bound for finite-dimensional linear programs. While the bound is not sharp in a pointwise sense, the result is non-asymptotic; i.e., holds for all~$\eps$ below a known threshold. Moreover, the constants are  known in terms of the data, which provided valuable intuition for our construction of  the rate function~$I$. One may also observe how the constants in~\cite{Weed.18} blow up as the cardinality of the support increases.

In the last decade, optimal transport has found myriad applications in machine learning, statistics, image processing, language processing, and other areas. The literature in the computational area has expanded very quickly and our account is highly incomplete; see~\cite{PeyreCuturi.19} for a recent monograph with extensive references. Exact computation of an optimal transport between marginals with $n$~atoms costs $O(n^3\log n)$,
prohibitive for modern applications with large data sets. The recent success of applied optimal transport is enabled by the advent of fast approximate solvers, and entropic regularization is among the most influential schemes for high-dimensional problems. Popularized by Cuturi~\cite{Cuturi.13} in this domain, it allows for the application of Sinkhorn's algorithm (also called iterative proportional fitting, and also due to Deming, Stephan, Fortet, Knopp and others) where each iteration is a matrix-vector multiplication costing $O(n^2)$. Importantly for modern applications, it is highly parallelizable on GPUs; a number of further advantages are highlighted in~\cite{BenamouCarlierCuturiNennaPeyre.15}. The convergence of this algorithm was rigorously discussed in~\cite{IrelandKullback.68,Ruschendorf.95}, among others. More recently, it was shown that $\delta$-accurate approximations of the transport cost can be  obtained in $\tilde{O}(n^2 /\delta)$ operations via entropic regularization; cf.~\cite{BlanchetJambulapatiKentSidford.18,LinHoJordan.19} and the references therein. %
In addition to computation accuracy, a second error in practice stems from sampling the marginals. For entropic optimal transport (with $\eps>0$ fixed), the rate of convergence of the empirical cost towards its population limit does not depend on the dimension, in contrast to the curse of dimensionality suffered by its unregularized counterpart~\cite{GenevayPeyreCuturi.18,MenaWeed.19}.
Addressing the combined problem, \cite{Berman.20} studies the convergence of the discrete Sinkhorn algorithm to an optimal transport potential in the joint limit when $\eps_{n}\to0$ and the marginals $\mu,\nu$ are approximated by discretizations $\mu_{n},\nu_{n}$ satisfying a certain density property. Explicit error bounds are derived, for instance for quadratic cost on the torus, yielding important insights into the optimal trade-off between~$n$ and~$\eps$.
In the present study, we focus on the discrepancy between the entropic optimizer~$\pi_{\eps}$ and the optimal transport~$\pi_{*}$ in a general setting and adopt a  local point of view.

Continuing with a different branch of related literature, recall that entropic optimal transport can also be phrased as the (static) Schr\"odinger bridge problem. Informally stated, consider a system of diffusing particles from time $t_{0}$ to $t_{1}$ in thermal equilibrium, and a given joint ``reference'' law~$R$ for its configuration at those times. If marginals $(\mu,\nu)$ differing from the ones of~$R$ are observed, what is the most likely evolution (joint law of~$\mu,\nu$) of the system conditional on~$R$? Schr\"odinger's answer amounts to $\pi^{*}=\argmin_{\Pi(\mu,\nu)} H(\cdot|R)$; see \cite{Follmer.88, Leonard.14} for extensive surveys including historical accounts. (This is the static formulation. Given the origins in physics, it is natural that much of the literature focuses on the \emph{dynamic} Schr\"odinger bridge problem, which asks for the dynamic evolution of the particle system over time $t\in[t_{0},t_{1}]$. The static problem is recovered by projecting to the marginals.)

The minimization of $H(\cdot|R)$ over $\Pi(\mu,\nu)$ coincides with the entropic optimal transport problem~\eqref{eq:EOTintro} if we introduce the cost function $c:=-\eps \log (\alpha^{-1} dR/d(\mu\otimes\nu))$, where the parameter $\eps>0$ is arbitrary and $\alpha$ is a normalizing constant (we tacitly assume that $R\sim \mu\otimes\nu$). Conversely, taking \eqref{eq:EOTintro} as the starting point, defining $R(\eps)$ by~$dR(\eps)/d(\mu\otimes\nu)=\alpha e^{-c/\eps}$ yields the associated Schr\"odinger bridge problem. Assuming for simplicity that $\{c=0\}$ is the graph of a function~$f:\X\to\Y$, Theorem~\ref{th:mainIntro} is then a large deviations principle as the reference measure~$R(\eps)$ degenerates to a deterministic coupling (meaning that a particle with given origin~$x$ travels to the predetermined destination~$f(x)$).\footnote{Schr\"odinger's ideas about the ``most likely evolution'' are usually presented as a large deviations result in the modern literature. That result is very different from the one just discussed.} This is also called the small-noise or small-time limit. While not pursued here, it seems plausible that a similar principle could be established for more general sequences~$R(\eps)$. 
From the point of view of Schr\"odinger bridges, another interesting follow-up question is whether a comparable large deviations result can be stated for the dynamic problem on path space.

Mikami~\cite{Mikami.02, Mikami.04} first highlighted the connection between Schr\"odinger equations and optimal transport in the small-noise limit; see also~\cite{ChenGeorgiouPavon.16} for a connection through a fluid dynamic formulation.
L\'eonard studied Schr\"odinger bridges in a series of works starting with~\cite{Leonard.01a,Leonard.01b}; see~\cite{Leonard.14} for further references. In~\cite{Leonard.12}, he established convergence of the value function to an optimal transport problem in the sense of Gamma-convergence for a general formulation of the problem. See also~\cite{CarlierDuvalPeyreSchmitzer.17} where a very accessible proof of the Gamma-convergence is presented for quadratic costs.
More recently, \cite{ConfortiTamanini.19,Pal.19} study the limit in specific settings and determine higher-order terms in the expansion of the Schr\"odinger (or entropic) value function around the optimal transport cost. These works complement earlier results of \cite{AdamsDirrPeletierZimmer.11,DuongLaschosRenger.13,ErbarMaasRenger.15} showing that the large deviation rate function for the empirical distribution of independent Brownian particles with drift is asymptotically equivalent to the Jordan--Kinderlehrer--Otto functional arising in the Wasserstein gradient flow. %
We mention that~\cite{ConfortiTamanini.19} also considers the large-time limit (corresponding to $\eps\to\infty$); cf.~\cite{ClercConfortiGentil.20} for recent developments. The setup in \cite{Pal.19} is closest to ours in that the entropic penalty and the limit~$\eps\to0$ are formulated in the same way, whereas the literature on Schr\"odinger bridges often formulates the zero-noise limit through a vanishing Laplacian. We also mention~\cite{GigliTamanini.21}, establishing convergence of the dual potentials for compact marginals (see~\cite{NutzWiesel.21} for a follow-up and more on the relation to the present work).

While the focus of the aforementioned works is on value functions and global quantities, the present study focuses on the local geometry and convergence. The value functions are not used at all, and so it is quite natural that the results hold even when costs are infinite. We are not aware of a large deviations principle similar to ours in the extant literature. 
One concrete example where these aspects are of interest, are the multidimensional ranks and quantiles that have been introduced in statistics to extend the usual scalar notions and familiar nonparametric tests; see~\cite{ChernozhukovEtAl.17,DebSen.19,delBarrioEtAl.18,GhosalSen.19}.
Here Brenier's map is fundamental, but like in the scalar case, moment conditions are not natural. McCann's geometric extension~\cite{McCann.95} of Brenier's map (see also \cite[pp.\,249--258]{Villani.09}) can be used to provide a definition irrespectively of the finiteness of the  value function. Unlike their scalar counterparts, the ranks defined through optimal transport are computationally expensive. Entropic optimal transport resolves that issue and provides an approximate Brenier's map. Leveraging this idea, a notion of ``differentiable ranks'' based on entropic optimal transport was recently proposed in~\cite{CuturiTeboulVert.19}. We expect that our results can be used to study the local deviations of these differentiable ranks from the unregularized ones.

Related to our technique in a broader sense, there have been recent works successfully using ideas of $c$-cyclical monotonicity outside the setting of classical optimal transport. Examples include martingale optimal transport~\cite{BeiglbockJuillet.12} and optimal Skorokhod embeddings~\cite{BeiglbockCoxHuesmann.14, BeiglbockNutzStebegg.21}. Finally, we mention the intriguing ``optimal entropy-transport problem'' studied in~\cite{LieroMielkeSavare.18}. Here, the usual optimal transport problem is relaxed in that the marginal constraints are replaced by an entropic penalty relative to a given pair of measures. While similar in name, this problem is quite different from ours, where the marginal constraints are strictly enforced and the entropy of the joint distribution is used as penalty.

\section{Cyclical Invariance}\label{se:cyclInvar}

Let $(\X,\mu)$ and $(\Y,\nu)$ be Polish probability spaces endowed with their Borel $\sigma$-fields and let $c: \X\times\Y\to\R_{+}$ be a measurable (cost) function. The associated optimal transport problem is 
\begin{equation}\label{eq:OT}
  \inf_{\pi\in\Pi(\mu,\nu)} \int_{\X\times\Y} c \,d\pi
\end{equation}
where $\Pi(\mu,\nu)$ is the set of all couplings; that is, probability measures~$\pi$ on $\X\times\Y$ with marginals $\mu=(\proj_{\X})_{\#}\pi$ and $\nu=(\proj_{\Y})_{\#}\pi$. Given a constant~$\eps>0$, the entropic optimal transport problem is
\begin{equation}\label{eq:EOT}
  \inf_{\pi\in\Pi(\mu,\nu)} \int_{\X\times\Y} c \,d\pi + \eps H(\pi|P), \quad P:=\mu\otimes\nu,
\end{equation}
where $H$ denotes the relative entropy or Kullback--Leibler divergence,
$$
  H(\pi|P):=\begin{cases}
\int \log (\frac{d\pi}{dP}) \,d\pi, & \pi\ll P,\\
\infty, & \pi\not\ll P.
\end{cases} 
$$
As detailed in Proposition~\ref{pr:factorization} of Appendix~\ref{se:factorization}, this problem admits a unique minimizer $\pi_{\eps}$ whenever the value~\eqref{eq:EOT} is finite; i.e., whenever
\begin{equation}\label{eq:finitenessCond}
\mbox{there exists $\pi\in\Pi(\mu,\nu)$ with $\int c \,d\pi+H(\pi|P)<\infty$.}
\end{equation}
Moreover, we then have $\pi_{\eps}\sim P$.

\begin{definition}\label{de:cyclInv}
  A coupling $\pi\in\Pi(\mu,\nu)$ is called \emph{$(c,\eps)$-cyclically invariant} if $\pi\sim P$ and its density admits a version $\frac{d\pi}{dP}: \X\times\Y\to(0,\infty)$ such that
\begin{equation}\label{eq:cyclInv}
  \prod_{i=1}^{k} \frac{d\pi}{dP}(x_{i},y_{i})=\exp\bigg(-\frac{1}{\eps}\bigg[\sum_{i=1}^{k}  c(x_{i},y_{i}) - \sum_{i=1}^{k}  c(x_{i},y_{i+1})\bigg]\bigg)\prod_{i=1}^{k} \frac{d\pi}{dP}(x_{i},y_{i+1})
\end{equation}  
 for all $k\in\N$ and $(x_{i},y_{i})_{i=1}^{k} \subset \X\times\Y$, where $y_{k+1}:=y_{1}$.
\end{definition} 

We omit the qualifier $(c,\eps)$ when there is no ambiguity. One elementary way to motivate Definition~\ref{de:cyclInv} is to derive a first-order condition of optimality for~\eqref{eq:EOT} through variational arguments in the case of discrete marginals, which indeed yields~\eqref{eq:cyclInv}. Cyclical invariance can be phrased more succinctly using the auxiliary reference measure $R=R(\eps)$ defined by the Gibbs kernel
\begin{equation}\label{eq:defR}
  \frac{dR}{dP} = \alpha e^{-c/\eps},
\end{equation}
where $\alpha=(\int e^{-c/\eps}\,dP)^{-1}$ is the normalizing constant. As $R\sim P$, we can state~\eqref{eq:cyclInv} as
\begin{equation}\label{eq:cyclInvR}
  \prod_{i=1}^{k} \frac{d\pi}{dR}(x_{i},y_{i})=\prod_{i=1}^{k} \frac{d\pi}{dR}(x_{i},y_{i+1}).
\end{equation}
This condition, in turn, is related to a multiplicative decomposition of the density~$d\pi/dR$; cf.\ Appendix~\ref{se:factorization}.
For our analysis of the limit $\eps\to0$, the less elegant definition~\eqref{eq:cyclInv} will be the more useful one, as it makes explicit the role of $\eps$ and links directly to the $c$-cyclical monotonicity condition of optimal transport.

\begin{proposition}\label{pr:cyclInv}
  (a) There is at most one $(c,\eps)$-cyclically invariant coupling $\pi\in\Pi(\mu,\nu)$.
  
  (b) Let~\eqref{eq:finitenessCond} hold.
  Then $\pi\in\Pi(\mu,\nu)$ is $(c,\eps)$-cyclically invariant if and only if it minimizes~\eqref{eq:EOT}. Moreover, there exists a unique such coupling.
\end{proposition} 

The proof is detailed in Appendix~\ref{se:factorization}. 
Under Condition~\eqref{eq:finitenessCond}, Proposition~\ref{pr:cyclInv} shows the equivalence between minimality and cyclical invariance. The notion of minimality is meaningful only under~\eqref{eq:finitenessCond}, otherwise all couplings have infinite cost. By contrast, we show in~\cite{GhosalNutzBernton.21b} that the notion of cyclical invariance remains meaningful in this context of infinite costs: existence and uniqueness hold under mild regularity conditions; e.g., when $\X,\Y$ are Euclidean spaces and~$c$ is continuous. %

In the remainder of this paper, we simply \emph{assume} that a $(c,\eps)$-cyclically invariant coupling $\pi_{\eps}\in\Pi(\mu,\nu)$~exists for every $\eps>0$, rather than imposing Condition~\eqref{eq:finitenessCond} as in much of the literature. One reason is that this condition precludes some applications of interest to us. In any event, the arguments in this paper do not simplify if~\eqref{eq:finitenessCond} is assumed.

\section{Cluster Points as $\eps\to0$}\label{se:clusterPoints}

Denote by $\pi_{\eps}$ the unique $(c,\eps)$-cyclically invariant coupling. In this section we show that cluster points of $\pi_{\eps}$ as $\eps\to0$ are $c$-cyclically monotone. The estimates leading to that conclusion are obtained by simply integrating the cyclical invariance condition. %

\begin{lemma}\label{le:expBound}
  Let $k\geq2$ and $0\leq\delta\leq\delta'\leq \infty$. Define
  $$
    A_{k}(\delta,\delta'):=\left\{(x_{i},y_{i})_{i=1}^{k}\in (\X\times\Y)^{k}: \delta\leq \sum_{i=1}^{k} c(x_{i},y_{i})-\sum_{i=1}^{k} c(x_{i},y_{i+1})\leq \delta'\right\}
  $$
  and let $A\subset A_{k}(\delta,\delta')$ be Borel. Then $\pi_{\eps}^{k}:=\prod_{i=1}^{k} \pi_{\eps}(dx_{i},dy_{i})$ satisfies
  \begin{equation}\label{eq:expBound}
    \pi_{\eps}^{k}(A)\leq e^{-\delta/\eps} \qforallq \eps>0.
  \end{equation}
  Suppose in addition that $\bar{A}:=\big\{(x_{i},y_{i+1})_{i=1}^{k}: (x_{i},y_{i})_{i=1}^{k}\in A\big\}$ satisfies $\liminf_{\eps\to0} \eps \log \pi_{\eps}^{k}(\bar{A})=0$. Then
  \begin{equation}\label{eq:expBoundReverse}
    \liminf_{\eps\to0} \eps \log  \pi_{\eps}^{k}(A) \geq -\delta'.
  \end{equation}  
\end{lemma} 

\begin{proof}
  Set $Z=d\pi_{\eps}/dP$. Using~\eqref{eq:cyclInv}, we have for $P^{k}$-a.e.\ $(x_{i},y_{i})_{i=1}^{k}\in A$ that
  \begin{align*}
    \prod Z(x_{i},y_{i}) 
    &=\exp\left\{-\eps^{-1} \big[\tsum c(x_{i},y_{i}) - \tsum c(x_{i},y_{i+1})\big] \right\} \prod Z(x_{i},y_{i+1}) \\
    &\leq e^{-\delta/\eps} \prod Z(x_{i},y_{i+1}).
  \end{align*}
  Integrating over $A$ with respect to $P^{k}=\prod P(dx_{i},dy_{i})=\prod P(dx_{i},dy_{i+1})$  yields 
  $$\pi_{\eps}^{k}(A)\leq e^{-\delta/\eps} \pi_{\eps}^{k}(\bar{A}) \leq e^{-\delta/\eps},$$
  which is~\eqref{eq:expBound}. Analogously, $\pi_{\eps}^{k}(A) \geq e^{-\delta'/\eps} \pi_{\eps}^{k}(\bar{A})$ and hence
  $$
    \eps \log  \pi_{\eps}^{k}(A) \geq -\delta' +\eps \log \pi_{\eps}^{k}(\bar{A}),
  $$
  so that~\eqref{eq:expBoundReverse} follows under the stated condition on $\bar{A}$.
\end{proof}

In all that follows, probability measures are considered with weak convergence; i.e., the topology induced by bounded continuous functions. We recall that $\Pi(\mu,\nu)$ is weakly compact; cf.\ \cite[p.\,45]{Villani.09}. As a consequence, any sequence of couplings admits at least one cluster point, and any cluster point is a coupling.  A set $\Gamma\subset \X\times\Y$ is called $c$-cyclically monotone if $\sum_{i=1}^{k} c(x_{i}, y_{i}) \leq \sum_{i=1}^{k} c(x_{i}, y_{i+1})$ for all $k\geq1$ and $(x_{i},y_{i}) \in \Gamma$, $1\leq i\leq k$.

\begin{proposition}\label{pr:limitsCyclMonotone}
    Let $c$ be continuous and let $\pi$ be a cluster point of $(\pi_{\eps})$ as $\eps\to0$. Then $\spt \pi$ is $c$-cyclically monotone, hence $\pi$ is an optimal transport as soon as the optimal transport problem~\eqref{eq:OT} is finite. If~\eqref{eq:OT} admits a unique $c$-cyclically monotone coupling $\pi_{*}\in\Pi(\mu,\nu)$, then  $\pi_{\eps}\to \pi_{*}$ as $\eps\to0$.
\end{proposition}

\begin{proof}
  Let $\eps_{n}\to0$ and $\pi_{\eps_{n}}\to\pi$. Suppose for contradiction that there are $(x_{i},y_{i})\in \spt \pi$, $1\leq i\leq k$ with $\sum_{i} c(x_{i},y_{i})> \sum_{i} c(x_{i},y_{i+1})$. By continuity there exist $\delta>0$ and open neighborhoods $U_{i}\ni (x_{i},y_{i})$ such that $\sum_{i} c(\tilde x_{i},\tilde y_{i}) \geq \delta+ \sum_{i} c(\tilde x_{i},\tilde y_{i+1})$ for all $(\tilde x_{i},\tilde y_{i})\in U_{i}$. Moreover, $\pi(U_{i})>0$ and hence $\liminf_{n} \pi_{\eps_{n}}(U_{i})>0$. On the other hand, $U_{1}\times\cdots\times U_{k}\subset A_{k}(\delta,\infty)$ implies $\pi_{\eps_{n}}^{k}(U_{1}\times\cdots\times U_{k})\to0$ by Lemma~\ref{le:expBound}, a contradiction. This shows that $\spt\pi$ is $c$-cyclically monotone. It is well known that cyclical monotonicity and optimality are equivalent when~\eqref{eq:OT} is finite; cf.\ \cite[Theorem~5.10, p.\,57]{Villani.09}. As $\Pi(\mu,\nu)$ is compact, $\pi_{\eps}$ must have cluster points as $\eps\to0$, so that uniqueness implies convergence.
\end{proof}

\begin{remark}\label{rk:limitsCyclMonotone}
 For the particular case of quadratic cost on~$\R^{d}$ and marginals satisfying certain integrability conditions, the conclusion of Proposition~\ref{pr:limitsCyclMonotone} is obtained in \cite{CarlierDuvalPeyreSchmitzer.17} by (arguably much more involved) Gamma-convergence arguments. That line of argument focuses on the properties of the value function, hence cannot be applied when the value function is infinite. A related but slightly different convergence result, also obtained by Gamma-convergence, is stated in \cite[Theorem~2.4]{Leonard.12} and includes lower semicontinuous cost functions. On the other hand, the convergence in Proposition~\ref{pr:limitsCyclMonotone} may fail if continuity is relaxed to lower semicontinuity: one example, discussed in more detail in~\cite[Remark~4.3]{Nutz.20}, is $c(x,y)=\1_{\{x\neq y\}}$ and $\mu=\nu=\Unif[0,1]$.
\end{remark} 

Uniqueness of $c$-cyclically monotone transports is known for many examples of continuous or semi-discrete optimal transport problems---arguably for most of the important examples except distance costs---and then Proposition~\ref{pr:limitsCyclMonotone} shows the convergence of $\pi_{\eps}$ as $\eps\to0$. See, e.g., \cite[Theorem~5.30, p.\,84]{Villani.09}. When the transport problem admits multiple solutions, it is not obvious whether $\pi_{\eps}$ converges. If there exists an optimal transport $\pi$ with $H(\pi|P)<\infty$,  one can show that $\pi_{\eps}$ converges to the unique optimal transport $\pi_{*}$ with minimal relative entropy $H(\cdot|P)$; cf.\ \cite[Theorem~5.1]{Nutz.20}. This includes the discrete case with finitely supported marginals as analyzed in~\cite{CominettiSanMartin.94}, but also the semi-discrete case (where one marginal is continuous) under minor integrability conditions. Convergence is also known for the scalar Monge problem where $c(x,y)=|x-y|$ on $\X=\Y=\R$ and the marginals are absolutely continuous; here a relatively explicit analysis is possible~\cite{DiMarinoLouet.18}. It has been conjectured that convergence holds in a general setting.

\section{Rate Function}\label{se:largeDeviations}

Throughout this section, the cost function~$c$ is assumed to be continuous. For simplicity of exposition, we shall also assume that
\begin{equation}\label{eq:convOT}
  \pi_{\eps}\to\pi_{*} \qmbox{as} \eps\to0,
\end{equation}
for some (necessarily $c$-cyclically monotone) transport $\pi_{*}\in\Pi(\mu,\nu)$. However, if it is merely known that $\pi_{\eps_{n}}\to\pi_{*}$ along a specific sequence $\eps_{n}\to0$, then all of our results hold along that sequence, regardless of whether $(\pi_{\eps})$ has other cluster points. In fact, the arguments in this paper are \emph{complementary} to the question of convergence discussed in the preceding paragraph: \emph{given} the convergence of a sequence, we describe the large deviations.

\subsection{Large Deviations Upper Bound}

In this subsection we introduce the function~$I$ and show  the large deviations upper bound; i.e., that~$I$ provides a lower bound for the large deviations rate. With the definitions in place, the arguments are straightforward and apply in great generality. 
We write $B_{r}(z)$ for the open ball of radius~$r$ around~$z$, in any metric space. The first lemma is a way to bound the decay of a ball in~$\X\times\Y$ based on the estimate for subsets of $(\X\times\Y)^{k}$ in Lemma~\ref{le:expBound}.

\begin{lemma}\label{le:expBoundWithSptPts}
  Let $(x,y)\in\X\times\Y$. Suppose there exist $(x_{i},y_{i})_{2\leq i\leq k}\subset\spt \pi_{*}$ with $k\geq2$ such that 
  $$
    \delta_{0}:= \sum_{i=1}^{k} c(x_{i}, y_{i}) - \sum_{i=1}^{k} c(x_{i}, y_{i+1}) >0, \quad \mbox{where}\quad (x_{1},y_{1}):=(x,y).
  $$ 
  Given $\delta<\delta_{0}$, there exist $\alpha,r,\eps_{0}>0$ such that
  $$
    \pi_{\eps}(B_{r}(x,y))\leq \alpha  e^{-\delta/\eps} \qforq \eps\leq\eps_{0}.
  $$
\end{lemma} 

\begin{proof}
Once again, continuity of $c$ implies that for $r>0$ small enough,
$\sum c(\tilde x_{i}, \tilde y_{i}) - \sum c(\tilde x_{i}, \tilde y_{i+1})\geq \delta$ for all $(\tilde x_{i},\tilde y_{i})  \in B_{i}:=B_{r}(x_{i},y_{i})$, and then $B_{1}\times\dots\times B_{k} \subset A_{k}(\delta,\infty)$ in Lemma~\ref{le:expBound} yields 
\begin{equation}\label{eq:proofexpBoundWithSptPts}
  \pi_{\eps}(B_{1}) \cdots \pi_{\eps}(B_{k})\leq e^{-\delta/\eps}.
\end{equation}
For $i\geq 2$ we have $\liminf \pi_{\eps}(B_{i})\geq \pi_{*}(B_{i})$ due to the weak convergence $\pi_{\eps}\to \pi_{*}$, and $\beta_{i}:=\pi_{*}(B_{i})>0$ as $(x_{i},y_{i})\in\spt \pi_{*}$. Let $\beta=\min_{i\geq2} \beta_{i}$. Then $\pi_{\eps}(B_{i})\geq\beta/2$ for $i\geq2$ and $\eps$ small, and  thus~\eqref{eq:proofexpBoundWithSptPts} yields $\pi_{\eps}(B_{1})\leq (\beta/2)^{1-k} e^{-\delta/\eps}$.
\end{proof}

Denote by $\Sigma(k)$ the set of permutations of $\{1,\dots,k\}$. Next, we state the definition of $I(x,y)$; it is designed to capture the rate~$\delta$ in Lemma~\ref{le:expBoundWithSptPts} and optimize it over the choice of~$(x_{i},y_{i})_{2\leq i\leq k}$.

\begin{lemma}\label{le:defI}
Given a $c$-cyclically monotone set $\emptyset\neq\Gamma\subseteq \X\times\Y$, define
\begin{equation}\label{eq:defI}
  I(x,y):=\sup_{k\geq2} \sup_{(x_{i},y_{i})_{i=2}^{k}\subset \Gamma} \sup_{\sigma\in\Sigma(k)} \, \sum_{i=1}^{k} c(x_{i},y_{i}) - \sum_{i=1}^{k}c(x_{i},y_{\sigma(i)})
\end{equation}
where $(x_{1},y_{1}):=(x,y)$.
Then $I: \X\times\Y \to [0,\infty]$ is lower semicontinuous and $I=0$ on $\Gamma$. We have 
\begin{equation}\label{eq:defI2}
    I(x,y)\geq\sup_{k\geq2} \sup_{(x_{i},y_{i})_{i=2}^{k}\subset \Gamma} \, \sum_{i=1}^{k} c(x_{i},y_{i}) - \sum_{i=1}^{k}c(x_{i},y_{i+1}),
\end{equation}
and equality holds as soon as $x\in\X_{0}:=\proj_{\X} \Gamma$ or $y\in\Y_{0}:=\proj_{\Y} \Gamma$.
\end{lemma}

\begin{proof}
  We have $I\geq0$ as $\sigma=\id$ is a possible choice in~\eqref{eq:defI}. For $(x,y)\in \Gamma$, the difference of sums in~\eqref{eq:defI} is nonpositive by cyclical monotonicity. The semicontinuity follows from the continuity of~$c$. 
  
  Let $I'(x,y)$ be the right-hand side of~\eqref{eq:defI2}. As the pairs $(x_{i},y_{i})_{i=2}^{k}$ can be relabeled arbitrarily, this is the same as~\eqref{eq:defI} except that the last supremum in~\eqref{eq:defI2} is taken over $\sigma\in\Sigma(k)\setminus\{\id\}$. If $I(x,y)>0$, the identity permutation is not optimal for the relevant pairs $(x_{i},y_{i})_{i=2}^{k}$ and equality must hold in~\eqref{eq:defI2}. Thus, if equality fails, then $I(x,y)=0$ whereas $I'(x,y)<0$. Let $x\in\X_{0}$, then we can choose $k=2$ and $(x_{2},y_{2})\in\Gamma$ with $x_{2}=x$, which yields $\sum_{i=1}^{2} c(x_{i},y_{i}) - \sum_{i=1}^{2}c(x_{i},y_{i+1})=0$ and hence $I'(x,y)\geq0$. The argument for $y\in\Y_{0}$ is symmetric.
\end{proof}

The reader may ignore the difference between~\eqref{eq:defI} and~\eqref{eq:defI2}; it is  merely a notational nuisance. We have the following result for the $c$-cyclically monotone set $\Gamma:=\spt \pi_{*}$,   also stated as Theorem~\ref{th:mainIntro}\,(a) in the Introduction.

\begin{corollary}\label{co:limsup}
  For any compact set $C\subset\X\times\Y$,
  $$
    \limsup_{\eps\to0} \eps \log \pi_{\eps}(C) \leq -\inf_{(x,y)\in C} I(x,y).
  $$ 
\end{corollary} 

\begin{proof}
  Fix $\eta>0$ and $(x,y)\in C$. By the definition of~$I(x,y)$ there are $k\geq 1$
  and $(x_{i},y_{i})_{i=2}^{k}\subset\Gamma$ such that
  $$
    \sum_{i=1}^{k} c(x_{i}, y_{i}) - \sum_{i=1}^{k} c(x_{i}, y_{i+1}) > I_{\eta}(x,y)-\eta/2,
  $$
  where  $(x_{1},y_{1}):=(x,y)$ and $I_{\eta}(x,y):= I(x,y)\wedge \eta^{-1}$. (The truncation is needed only if $I(x,y)=\infty$.)
  Lemma~\ref{le:expBoundWithSptPts} thus yields a ball $B_{r}(x,y)$ with
  \begin{equation}\label{eq:limsupProof}
    \limsup \eps \log \pi_{\eps}(B_{r}(x,y)) \leq - I_{\eta}(x,y) + \eta.
  \end{equation}
  This holds for every $(x,y)\in C$, and as~$C$ is covered by finitely many such balls, we deduce that
  $$
    \limsup \eps \log \pi_{\eps}(C) \leq -\inf_{(x,y)\in C} I_{\eta}(x,y) + \eta.
  $$
  Recalling that $\eta>0$ was arbitrary, the claim follows.
\end{proof} 

We note that the measure  $\pi_{*}=\lim_{\eps}\pi_{\eps}$ is not compactly supported in general. It is then an open problem how to relax the compactness condition in Corollary~\ref{co:limsup} and hence obtain a ``stronger'' version of the large deviations principle.

\subsection{Large Deviations Lower Bound}

Our next aim is to show that $I$ is also an upper bound for the large deviations rate, thus matching the bound in Corollary~\ref{co:limsup}. This will be accomplished in two slightly different settings and approaches. The dual approach expresses~$I$ as the gap~\eqref{eq:IandPotentials} between the cost~$c$ and the solution of the dual optimal transport problem, whereas the primal directly uses the definition~\eqref{eq:defI} of~$I$ and imposes regularity conditions. The results correspond to Theorem~\ref{th:mainIntro}\,(b) in the Introduction.

\subsubsection{Bound via Kantorovich Potential}
  
We start with the dual approach, first recalling some standard notions of optimal transport---we have tried to consistently use the notation of~\cite{Villani.09}. A proper function $\psi: \X\to(-\infty,\infty]$ is called $c$-convex if there exists some $\zeta: \Y\to[-\infty,\infty]$ such that $\psi(x)=\sup_{y\in\Y} [\zeta(y)-c(x,y)]$ for all $x\in\X$. Its $c$-conjugate is defined by $\psi^{c}(y):=\inf_{x\in\X} [\psi(x)+c(x,y)]$ for $y\in\Y$, and its $c$-subdifferential is
$$
  \partial_{c}\psi=\{(x,y)\in\X\times\Y:\, \psi^{c}(y) -\psi(x)=c(x,y)\}.
$$
Given a $c$-cyclically monotone set $\Gamma$, a $c$-convex function $\psi$ is called a Kantorovich potential if $\Gamma\subset\partial_{c}\psi$; that is, if $\psi^{c}(y) -\psi(x)=c(x,y)$ on $\Gamma$. This implies in particular that $\psi,\psi^{c}$ are finite on
$$
  \X_{0} := \proj_{\X}\Gamma, \quad \Y_{0} := \proj_{\Y}\Gamma.
$$
In the context of optimal transport, $\spt \pi\subset\partial_{c}\psi$ for some optimal $\pi\in\Pi(\mu,\nu)$ implies that $\partial_{c}\psi$ contains the support of any optimal transport. Indeed, $\partial_{c}\psi$ is a maximal $c$-monotone set for inclusion. In what follows, the cyclically monotone set of interest is $\Gamma=\spt\pi_{*}$, where $\pi_{*}$ is the limiting optimal transport~\eqref{eq:convOT}.

\begin{assumption}\label{as:uniquePotential}
  Uniqueness of Kantorovich potentials holds on~$\X_{0}$; that is, for any $c$-convex functions $\psi_{1},\psi_{2}$ on~$\X$ with $\Gamma\subset\partial_{c}\psi_{i}$, it holds that $\psi_{1}-\psi_{2}$ is constant on~$\X_{0}$.
\end{assumption}

This is often considered a fairly weak assumption, at least for differentiable cost functions, and we detail sufficient conditions in Proposition~\ref{pr:dualUniqueness} of Appendix~\ref{se:uniquenessOfPotentials}. However, we emphasize that connectedness of at least one marginal support is crucial (cf.\ Example~\ref{ex:2times2} below).

As announced, Assumption~\ref{as:uniquePotential} allows us to express~$I$ through the Kantorovich potential; see~\eqref{eq:IandPotentials}. For our present purpose, the key consequence is~\eqref{eq:Isum}. It is worth noting that~\eqref{eq:IandPotentials} also allows us to translate a large body of known results about $c$-convex functions, such as regularity results, into statements about~$I$. Finally, the gap~\eqref{eq:IandPotentials} also plays a role in the regularity theory of optimal transport maps (especially in~\cite{Loeper.09}), thus relating to the second approach in Section~\ref{se:upperBoundViaRegularity} below. 

\begin{proposition}\label{pr:dualIndep}
  Let Assumption~\ref{as:uniquePotential} hold. Then 
  \begin{equation}\label{eq:IandPotentials}
      I(x,y)=c(x,y)- \psi^{c}(y)+\psi(x), \quad (x,y)\in \X_{0}\times\Y_{0}
  \end{equation}
  for any Kantorovich potential~$\psi$. In particular, $I<\infty$ on $\X_{0}\times\Y_{0}$.
  If $(x,y),(x',y')\in\X_{0}\times\Y_{0}$ are such that $(x',y), (x,y')\in\Gamma$, then
\begin{align}\label{eq:Isum}
  I(x,y)+ I(x',y')= c(x,y)+ c(x',y')- c(x,y')-c(x',y). 
\end{align}
\end{proposition}

\begin{proof}
We first elaborate on Assumption~\ref{as:uniquePotential}. A particular family of Kantorovich potentials, sometimes called Rockafellar antiderivatives of $\Gamma$, is defined as follows (cf.\ \cite[Equation~(5.17), p.\,65]{Villani.09}): fix $(x_{0},y_{0})\in\Gamma$ and set
\begin{equation}\label{eq:antiderivative}
  \psi_{(x_{0},y_{0})}(x):= \sup_{k\geq1} \sup_{(x_{i},y_{i})_{i=1}^{k}\in\Gamma}  \sum_{i=0}^{k} [c(x_{i},y_{i}) - c(x_{i+1},y_{i})] , \quad\mbox{where }x_{k+1}:=x. 
\end{equation}
It then holds that $\psi_{(x_{0},y_{0})}(x)=0$ for $x=x_{0}$. Clearly Assumption~\ref{as:uniquePotential} implies that changing the reference point $(x_{0},y_{0})$ only changes this potential by a constant. In particular,
\begin{equation}\label{eq:Psi}
  \Psi_{(x_{0},y_{0})}(x,y):=\psi^{c}_{(x_{0},y_{0})}(y)-\psi_{(x_{0},y_{0})}(x), \quad (x,y)\in\X_{0}\times\Y_{0}
\end{equation}
does not depend on $(x_{0},y_{0})\in\Gamma$, and we may simply write $\Psi:=\Psi_{(x_{0},y_{0})}$. Indeed, under Assumption~\ref{as:uniquePotential}, $\Psi$ is even the same for any potential~$\psi$.

We now use this independence to prove the lemma. To avoid notational conflict, we first rewrite the definition~\eqref{eq:antiderivative} as
\begin{equation}\label{eq:antiderivativeRewrite}
  \psi_{(\bar{x},\bar{y})}(x)= \sup_{k\geq2} \sup_{(x_{i},y_{i})_{i=2}^{k}\in\Gamma}  c(\bar{x},\bar{y}) + \sum_{i=2}^{k} [c(x_{i},y_{i}) - c(x_{i+1},y_{i})] - c(x_{2},\bar{y}),
\end{equation}
where we have avoided the subscript $i=1$.
Fix $(x,y)\in \X_{0}\times\Y_{0}$. Writing $(x_{1},y_{1}):=(x,y)$ as in Lemma~\ref{le:defI}, the definition $x_{k+1}:=x$ of~\eqref{eq:antiderivative} becomes our usual cyclical convention $x_{k+1}=x_{1}$.
As $y\in\Y_{0}$, there exists $\bar{x}\in\X_{0}$ such that $(\bar{x},y)\in\Gamma$. Using~\eqref{eq:antiderivativeRewrite} with $\bar{y}:=y$ then yields
\begin{align*}
  \psi_{(\bar{x},y)}(x) 
  &= \sup_{k\geq2} \sup_{(x_{i},y_{i})_{i=2}^{k}\in\Gamma}  c(\bar{x},y) + \sum_{i=2}^{k} c(x_{i},y_{i}) - \sum_{i=2}^{k} c(x_{i+1},y_{i}) - c(x_{2},y) \\
  &= \sup_{k\geq2} \sup_{(x_{i},y_{i})_{i=2}^{k}\in\Gamma}  c(\bar{x},y) - c(x,y)+ \sum_{i=1}^{k} c(x_{i},y_{i}) - \sum_{i=1}^{k} c(x_{i+1},y_{i})\\
  &=  c(\bar{x},y) - c(x,y) + \sup_{k\geq2} \sup_{(x_{i},y_{i})_{i=2}^{k}\in\Gamma} \sum_{i=1}^{k} c(x_{i},y_{i}) - \sum_{i=1}^{k} c(x_{i},y_{i+1})\\
    & = c(\bar{x},y) - c(x,y) + I(x,y),
\end{align*}
where we have used the last part of Lemma~\ref{le:defI}. %
  In view of $\psi_{(\bar x,y)}(\bar x)=0$, the fact that $\Psi_{(\bar x,y)}=c$ on~$\Gamma$ shows in particular that $c(\bar x,y)=\psi^{c}_{(\bar x,y)}(y)$, and hence the preceding display yields
  $$
    I(x,y)=c(x,y)+\psi_{(\bar x,y)}(x) - \psi^{c}_{(\bar x,y)}(y)=c(x,y)-\Psi_{(\bar x,y)}(x,y).
  $$
  By the first part of the proof, $\Psi_{(\bar x,y)}(\cdot)=\Psi(\cdot)$ does not depend on~$(\bar x,y)$ and the above is precisely~\eqref{eq:IandPotentials}. 
  
  To see~\eqref{eq:Isum}, let $(x',y), (x,y')\in\Gamma$.
  Using that $I=0$ on $\Gamma$ by Lemma~\ref{le:defI} and then~\eqref{eq:IandPotentials},
  \begin{align*}
    I(x,y)+ I(x',y')&= I(x,y)+ I(x',y') - I(x,y')- I(x',y) \\
    &= c(x,y)+ c(x',y') - c(x,y')- c(x',y) \\
    & \quad-\Psi(x,y)- \Psi(x',y') + \Psi(x,y')+ \Psi(x',y),
  \end{align*}
  where the last line vanishes as $\Psi$ is a sum of marginal functions~\eqref{eq:Psi}. %
\end{proof}

\begin{remark}\label{rk:onDualIndepAssumption}
  The proof of Proposition~\ref{pr:dualIndep} is based on the condition that 
  \begin{equation}\label{eq:indepCond}
    \mbox{$\Psi_{(x_{0},y_{0})}$ of~\eqref{eq:Psi} does not depend on $(x_{0},y_{0})\in\Gamma$}, 
  \end{equation}
  which may seem weaker than Assumption~\ref{as:uniquePotential}. However, Assumption~\ref{as:uniquePotential} is in fact equivalent to~\eqref{eq:indepCond}; the proof is stated below. As a direct consequence, another equivalent condition is that the Rockafellar antiderivative~\eqref{eq:antiderivative} be independent of~$(x_{0},y_{0})$. The symmetry of~\eqref{eq:indepCond} shows that it is further equivalent to impose the analogue of Assumption~\ref{as:uniquePotential} on~$\Y$ instead of~$\X$.
\end{remark}

\begin{proof}[Proof that~\eqref{eq:indepCond} implies Assumption~\ref{as:uniquePotential}.]
  By construction, the Rockafellar antiderivative $\psi_{0}:=\psi_{(x_{0},y_{0})}$ of~\eqref{eq:antiderivative} has the minimality property $\psi_{0} \leq \xi$ on~$\X_{0}$ whenever~$\xi$ is a potential with $\xi(x_{0})=0=\psi_{0}(x_{0})$. (See \cite[p.\,62]{Villani.09}, or \cite{BartzReich.12} for a more general result and further context.) Consider another point~$(x_{1},y_{1})\in\Gamma$, let $\psi_{1}:=\psi_{(x_{1},y_{1})}$ and let~$\xi$ be any potential. Using the minimality twice,
  $$
    \psi_{0}(x_{1})-\psi_{0}(x_{0}) \leq \xi(x_{1}) - \xi(x_{0}) \leq \psi_{1}(x_{1})-\psi_{1}(x_{0}).
  $$
  Given~\eqref{eq:indepCond}, the right-hand side can be expressed as 
  \begin{align*}
    \psi_{1}(x_{1})-\psi_{1}(x_{0})
     & = \psi_{1}(x_{1}) - \psi_{1}^{c}(y_{0})- \psi_{1}(x_{0}) + \psi_{1}^{c}(y_{0}) \\
     & = \psi_{0}(x_{1}) - \psi_{0}^{c}(y_{0})- \psi_{0}(x_{0}) + \psi_{0}^{c}(y_{0}) \\
     & = \psi_{0}(x_{1})- \psi_{0}(x_{0}),
  \end{align*} 
  which is the left-hand side. It follows that $\psi_{0}(x_{1})-\psi_{0}(x_{0}) = \xi(x_{1}) - \xi(x_{0})$ for any potential~$\xi$, and as $x_{0},x_{1}\in\X_{0}$ were arbitrary, Assumption~\ref{as:uniquePotential} holds.
\end{proof}

We can now show the large deviations lower bound.

\begin{corollary}\label{co:liminf}
  Let Assumption~\ref{as:uniquePotential} hold. For any open set $U\subset \X_{0}\times\Y_{0}$,
  $$
    \liminf_{\eps\to0} \eps \log \pi_{\eps}(U) \geq -\inf_{(x,y)\in U} I(x,y).
  $$ 
\end{corollary}  

\begin{proof}
It suffices to show that given $(x,y)\in U$ and $\eta>0$, there exists $r_0>0$ such that for all $r<r_0$,
 \begin{align*}%
   \limsup_{\eps\to 0} - \eps \log \pi_{\eps}(B_{r}(x,y))\leq I(x,y) + \eta.
 \end{align*}
Let $\eta>0$, pick any $(x',y')\in\X_{0}\times\Y_{0}$ such that $(x',y),(x,y')\in\Gamma$, and set
$$
  \alpha:=c(x,y)+ c(x', y')- c(x,y')-c(x',y).
$$
We have $I(x,y)<\infty$ and $I(x',y')<\infty$ by Proposition~\ref{pr:dualIndep}.
For $r>0$ small enough we may use Lemma~\ref{le:expBound} with $\delta':=\alpha + \eta/2$  and $B_{r}(x,y)\times B_{r}(x',y')\subset A_{2}(0,\delta')$ to obtain 
  \begin{align}
\limsup \, & -\eps \left[ \log \pi_{\eps}(B_{r}(x,y))+ \log \pi_{\eps}(B_{r}(x',y')\right]\nonumber\\
&= \limsup - \eps \log \pi_{\eps}^{2}\big(B_{r}(x,y)\times B_{r}(x',y')\big) \nonumber\\
&\leq \alpha + \eta/2. \label{eq:proofLimInf0}
\end{align} 
On the other hand, for $r$ small enough, Lemma~\ref{le:expBoundWithSptPts} yields as in~\eqref{eq:limsupProof} that
  \begin{align}
    \liminf -\eps \log \pi_{\eps}(B_{r}(x',y')) & \geq I(x',y') - \eta/2. \label{eq:proofLimInf2}
  \end{align} 
  Using~\eqref{eq:proofLimInf0}, then~\eqref{eq:Isum} and finally~\eqref{eq:proofLimInf2},
  \begin{align*}
 \limsup &-\eps \log \pi_{\eps}(B_{r}(x,y)) + \liminf -\eps \log \pi_{\eps}(B_{r}(x',y'))\\
    & \leq \limsup \, -\eps \left[ \log \pi_{\eps}(B_{r}(x,y))+ \log \pi_{\eps}(B_{r}(x',y')\right] \\
    &\leq \alpha+ \eta/2 \\
    &= I(x,y) + I(x',y') + \eta/2 \\
    & \leq I(x,y) + \liminf -\eps \log \pi_{\eps}(B_{r}(x',y')) + \eta %
\end{align*}
  and the claim follows. 
\end{proof}

The following simple example shows that if both marginals supports are disconnected (and Assumption~\ref{as:uniquePotential} is violated), $I$ may fail to be an upper bound for the rate function.

\begin{example}[Disconnected Supports]\label{ex:2times2}
Consider the normalized $2\times 2$ assignment problem: $\X=\Y=\{1,2\}$ and $\mu=\nu=(\delta_{\{1\}}+\delta_{\{2\}})/2$. Here $\Pi(\mu,\nu)$ is the convex hull of the two couplings
$$
  \pi_{*}=(\delta_{\{(1,1)\}}+\delta_{\{(2,2)\}})/2, \quad \pi_{0}=(\delta_{\{(1,2)\}}+\delta_{\{(2,1)\}})/2.
$$
In particular, every $\pi\in\Pi(\mu,\nu)$ is symmetric: $\pi\{(1,2)\}=\pi\{(2,1)\}$. Consider  a cost function~$c$ with $c(1,1)=c(2,2)=0$ and $c(1,2)+c(2,1)>0$. Then $\pi_{*}$ is the unique optimal transport and we know that $\pi_{\eps}\to\pi_{*}$. Let 
\begin{equation}\label{eq:discreteExTrueRate}
  r(i,j):=\lim_{\eps\to0} -\eps\log \pi_{\eps}(\{i,j\})
\end{equation}
be the exponential rate of convergence. Using Lemma~\ref{le:expBound} with 
$$
  A=\{(1,2),(2,1)\}\subset A_{2}(\delta,\delta)
$$
for $\delta:=c(1,2)+c(2,1)>0$ shows $r(1,2)+r(2,1) = \delta$. As $\pi_{\eps}$ must be symmetric, we conclude that the true exponential rate  is
$$
  r(1,2)=r(2,1) = \delta/2.
$$
(A priori, it may not be obvious that the limit~\eqref{eq:discreteExTrueRate} exists, but a posteriori, this is justified as every subsequential limit leads to the same value.)
On the other hand, the definition~\eqref{eq:defI} of~$I$ readily yields that
$I \equiv 0$.
\end{example}

\subsubsection{Bound via Regularity}\label{se:upperBoundViaRegularity}

In the remainder of the section we present an alternative approach to the large deviations lower bound which does not (directly) refer to potentials but instead employs a continuity condition for the limiting optimal transport~$\pi_{*}$. We call a subset of a metric space arcwise connected if any two points are connected by a continuous curve of finite length.

\begin{assumption}\label{as:crossContinuity}
 (a) $\Gamma=\graph T$ for a map $T: \X_{0}\to\Y$.

 (b) $\X_{0}$ is arcwise connected.
 
 (c) The function $c(\cdot, T(\cdot))$ has the following continuity property: given a compact $K\subset\X_{0}$, we have uniformly over $x_{1},x_{2}\in K$ that
\begin{equation}\label{eq:crossContinuity}
  |c(x_{1},T(x_{1}))+ c(x_{2},T(x_{2}))-c(x_{1},T(x_{2}))- c(x_{2},T(x_{1}))| = o(d(x_{1},x_{2})).
\end{equation}
\end{assumption}

As an example, consider $\X=\Y=\R^{d}$ with cost $c(x,y)=\|x-y\|^{2}/2$ and an optimal transport $\pi$ given by a continuous transport map $T$ on the arcwise connected support $\spt\mu$. Then Assumption~\ref{as:crossContinuity} holds with $\X_{0}=\spt\mu$, as~\eqref{eq:crossContinuity} equals
$$
  |\br{x_{1}-x_{2},T(x_{1})-T(x_{2})}| \leq \|x_{1}-x_{2}\|\|T(x_{1})-T(x_{2})\|
$$
and $T$ is uniformly continuous on compact sets. General sufficient conditions for the continuity of~$T$ can be found in~\cite[Theorem~1]{CorderoFigalli.19}.

Next, we show how to establish the key half of~\eqref{eq:Isum} under Assumption~\ref{as:crossContinuity}.

\begin{lemma}\label{le:sumI2}
  Let Assumption~\ref{as:crossContinuity} hold. If $(x,y),(x',y')\in\X_{0}\times\Y_{0}$ are such that $(x',y), (x,y')\in\Gamma$, then
\begin{align}\label{eq:Isum2}
  I(x,y)+ I(x',y') \geq c(x,y)+ c(x',y')- c(x,y')-c(x',y). 
\end{align}
\end{lemma} 

\begin{proof}
  Set $(x_{1},y_{1}):=(x,y)$ and $(x'_{1},y'_{1}):=(x',y')$. Let $k\geq 2$ and consider arbitrary $(x_{i},y_{i}), (x'_{i},y'_{i})\in\Gamma$ for $2\leq i\leq k$. The definition of~$I$ yields that
  \begin{align*}
      I(x,y)+ I(x',y')\geq \sum_{i=1}^{k} [c(x_{i},y_{i}) + c(x'_{i},y'_{i})] - \sum_{i=1}^{k} [c(x_{i},y_{i+1}) + c(x'_{i},y'_{i+1})].
  \end{align*} 
  This holds in particular for the choices $x_{k}:=x'$ and $x'_{k}:=x$, which entail that $y_{k}=T(x')=y$ and $y'_{k}=T(x)=y'$. Moreover, we have $y_{i}=T(x_{i})$ and $y'_{i}=T(x'_{i})$ for $i\geq2$. Separating the first term of the first sum and the last term of the second sum, we obtain that
  \begin{align*}
      I(x,y)+ I(x',y') &\geq c(x,y)+ c(x',y')- c(x,y')-c(x',y)\\
       &\quad  + \sum_{i=2}^{k} [c(x_{i},y_{i}) + c(x'_{i},y'_{i})] - \sum_{i=1}^{k-1} [c(x_{i},y_{i+1}) + c(x'_{i},y'_{i+1})].
  \end{align*} 
  We further choose $x'_{i}:=x_{k-i+1}$ for $i=2,\dots,k-1$, which implies $y'_{i}=y_{k-i+1}$ for $i=2,\dots,k-1$. Then the first sum can be rearranged as
  \begin{align}\label{eq:firstSum}
  \sum_{i=2}^{k} c(x_{i},y_{i}) + c(x'_{i},y'_{i})
  =\sum_{i=1}^{k-1} c(x_{i},T(x_{i}))+ c(x_{i+1},T(x_{i+1}))
  \end{align}
  and the second sum can be rearranged as
  \begin{align}\label{eq:secondSum}
  \sum_{i=1}^{k-1} c(x_{i},y_{i+1}) + c(x'_{i},y'_{i+1})
  =\sum_{i=1}^{k-1} c(x_{i},T(x_{i+1}))+ c(x_{i+1},T(x_{i}));
  \end{align}
  (These rearrangements are elementary if tedious; Figure~\ref{fi:sums} may be helpful to complete them.)
\begin{figure}[tb]
\begin{center}
\resizebox{\textwidth}{!}{

\tikzset{every picture/.style={line width=0.75pt}} %

\begin{tikzpicture}[x=0.75pt,y=0.75pt,yscale=-1,xscale=1]
\draw    (25.5,110.13) -- (295.94,110.13) ;
\draw    (25.5,220.45) -- (295.94,220.45) ;
\draw [fill={rgb, 255:red, 255; green, 255; blue, 255 }  ,fill opacity=1 ] [dash pattern={on 4.5pt off 4.5pt}]  (59.74,110) -- (59.74,148) -- (59.74,220.26) ;
\draw [shift={(59.74,220.26)}, rotate = 90] [color={rgb, 255:red, 0; green, 0; blue, 0 }  ][fill={rgb, 255:red, 0; green, 0; blue, 0 }  ][line width=0.75]      (0, 0) circle [x radius= 3.35, y radius= 3.35]   ;
\draw [shift={(59.74,110)}, rotate = 90] [color={rgb, 255:red, 0; green, 0; blue, 0 }  ][fill={rgb, 255:red, 0; green, 0; blue, 0 }  ][line width=0.75]      (0, 0) circle [x radius= 3.35, y radius= 3.35]   ;
\draw [fill={rgb, 255:red, 255; green, 255; blue, 255 }  ,fill opacity=1 ] [dash pattern={on 4.5pt off 4.5pt}]  (111.09,110.39) -- (111.09,148.39) -- (111.09,220.65)(108.09,110.39) -- (108.09,148.39) -- (108.09,220.65) ;
\draw [shift={(109.59,220.65)}, rotate = 90] [color={rgb, 255:red, 0; green, 0; blue, 0 }  ][fill={rgb, 255:red, 0; green, 0; blue, 0 }  ][line width=0.75]      (0, 0) circle [x radius= 3.35, y radius= 3.35]   ;
\draw [shift={(109.59,110.39)}, rotate = 90] [color={rgb, 255:red, 0; green, 0; blue, 0 }  ][fill={rgb, 255:red, 0; green, 0; blue, 0 }  ][line width=0.75]      (0, 0) circle [x radius= 3.35, y radius= 3.35]   ;
\draw [fill={rgb, 255:red, 255; green, 255; blue, 255 }  ,fill opacity=1 ] [dash pattern={on 4.5pt off 4.5pt}]  (160.94,109.79) -- (160.94,147.79) -- (160.94,220.05)(157.94,109.79) -- (157.94,147.79) -- (157.94,220.05) ;
\draw [shift={(159.44,220.05)}, rotate = 90] [color={rgb, 255:red, 0; green, 0; blue, 0 }  ][fill={rgb, 255:red, 0; green, 0; blue, 0 }  ][line width=0.75]      (0, 0) circle [x radius= 3.35, y radius= 3.35]   ;
\draw [shift={(159.44,109.79)}, rotate = 90] [color={rgb, 255:red, 0; green, 0; blue, 0 }  ][fill={rgb, 255:red, 0; green, 0; blue, 0 }  ][line width=0.75]      (0, 0) circle [x radius= 3.35, y radius= 3.35]   ;
\draw [fill={rgb, 255:red, 255; green, 255; blue, 255 }  ,fill opacity=1 ] [dash pattern={on 4.5pt off 4.5pt}]  (210.79,110) -- (210.79,148) -- (210.79,220.26)(207.79,110) -- (207.79,148) -- (207.79,220.26) ;
\draw [shift={(209.29,220.26)}, rotate = 90] [color={rgb, 255:red, 0; green, 0; blue, 0 }  ][fill={rgb, 255:red, 0; green, 0; blue, 0 }  ][line width=0.75]      (0, 0) circle [x radius= 3.35, y radius= 3.35]   ;
\draw [shift={(209.29,110)}, rotate = 90] [color={rgb, 255:red, 0; green, 0; blue, 0 }  ][fill={rgb, 255:red, 0; green, 0; blue, 0 }  ][line width=0.75]      (0, 0) circle [x radius= 3.35, y radius= 3.35]   ;
\draw [fill={rgb, 255:red, 255; green, 255; blue, 255 }  ,fill opacity=1 ] [dash pattern={on 4.5pt off 4.5pt}]  (259.15,109.39) -- (259.15,147.39) -- (259.15,219.65) ;
\draw [shift={(259.15,219.65)}, rotate = 90] [color={rgb, 255:red, 0; green, 0; blue, 0 }  ][fill={rgb, 255:red, 0; green, 0; blue, 0 }  ][line width=0.75]      (0, 0) circle [x radius= 3.35, y radius= 3.35]   ;
\draw [shift={(259.15,109.39)}, rotate = 90] [color={rgb, 255:red, 0; green, 0; blue, 0 }  ][fill={rgb, 255:red, 0; green, 0; blue, 0 }  ][line width=0.75]      (0, 0) circle [x radius= 3.35, y radius= 3.35]   ;
\draw    (363.06,109.74) -- (633.5,109.74) ;
\draw    (363.06,220.05) -- (633.5,220.05) ;
\draw  [dash pattern={on 4.5pt off 4.5pt}]  (447.15,110) -- (397.3,219.86) ;
\draw [shift={(397.3,219.86)}, rotate = 114.41] [color={rgb, 255:red, 0; green, 0; blue, 0 }  ][fill={rgb, 255:red, 0; green, 0; blue, 0 }  ][line width=0.75]      (0, 0) circle [x radius= 3.35, y radius= 3.35]   ;
\draw [shift={(447.15,110)}, rotate = 114.41] [color={rgb, 255:red, 0; green, 0; blue, 0 }  ][fill={rgb, 255:red, 0; green, 0; blue, 0 }  ][line width=0.75]      (0, 0) circle [x radius= 3.35, y radius= 3.35]   ;
\draw  [dash pattern={on 4.5pt off 4.5pt}]  (397.3,109.61) -- (447.15,220.26) ;
\draw [shift={(397.3,109.61)}, rotate = 65.75] [color={rgb, 255:red, 0; green, 0; blue, 0 }  ][fill={rgb, 255:red, 0; green, 0; blue, 0 }  ][line width=0.75]      (0, 0) circle [x radius= 3.35, y radius= 3.35]   ;
\draw  [dash pattern={on 4.5pt off 4.5pt}]  (497,110.39) -- (447.15,220.26) ;
\draw [shift={(447.15,220.26)}, rotate = 114.41] [color={rgb, 255:red, 0; green, 0; blue, 0 }  ][fill={rgb, 255:red, 0; green, 0; blue, 0 }  ][line width=0.75]      (0, 0) circle [x radius= 3.35, y radius= 3.35]   ;
\draw  [dash pattern={on 4.5pt off 4.5pt}]  (447.15,110) -- (497,220.65) ;
\draw [shift={(447.15,110)}, rotate = 65.75] [color={rgb, 255:red, 0; green, 0; blue, 0 }  ][fill={rgb, 255:red, 0; green, 0; blue, 0 }  ][line width=0.75]      (0, 0) circle [x radius= 3.35, y radius= 3.35]   ;
\draw  [dash pattern={on 4.5pt off 4.5pt}]  (546.61,110) -- (496.76,219.86) ;
\draw [shift={(496.76,219.86)}, rotate = 114.41] [color={rgb, 255:red, 0; green, 0; blue, 0 }  ][fill={rgb, 255:red, 0; green, 0; blue, 0 }  ][line width=0.75]      (0, 0) circle [x radius= 3.35, y radius= 3.35]   ;
\draw [shift={(546.61,110)}, rotate = 114.41] [color={rgb, 255:red, 0; green, 0; blue, 0 }  ][fill={rgb, 255:red, 0; green, 0; blue, 0 }  ][line width=0.75]      (0, 0) circle [x radius= 3.35, y radius= 3.35]   ;
\draw  [dash pattern={on 4.5pt off 4.5pt}]  (496.76,109.61) -- (546.61,220.26) ;
\draw [shift={(496.76,109.61)}, rotate = 65.75] [color={rgb, 255:red, 0; green, 0; blue, 0 }  ][fill={rgb, 255:red, 0; green, 0; blue, 0 }  ][line width=0.75]      (0, 0) circle [x radius= 3.35, y radius= 3.35]   ;
\draw  [dash pattern={on 4.5pt off 4.5pt}]  (596.46,110.39) -- (546.61,220.26) ;
\draw [shift={(546.61,220.26)}, rotate = 114.41] [color={rgb, 255:red, 0; green, 0; blue, 0 }  ][fill={rgb, 255:red, 0; green, 0; blue, 0 }  ][line width=0.75]      (0, 0) circle [x radius= 3.35, y radius= 3.35]   ;
\draw [shift={(596.46,110.39)}, rotate = 114.41] [color={rgb, 255:red, 0; green, 0; blue, 0 }  ][fill={rgb, 255:red, 0; green, 0; blue, 0 }  ][line width=0.75]      (0, 0) circle [x radius= 3.35, y radius= 3.35]   ;
\draw  [dash pattern={on 4.5pt off 4.5pt}]  (546.61,110) -- (596.46,220.65) ;
\draw [shift={(596.46,220.65)}, rotate = 65.75] [color={rgb, 255:red, 0; green, 0; blue, 0 }  ][fill={rgb, 255:red, 0; green, 0; blue, 0 }  ][line width=0.75]      (0, 0) circle [x radius= 3.35, y radius= 3.35]   ;
\draw [shift={(546.61,110)}, rotate = 65.75] [color={rgb, 255:red, 0; green, 0; blue, 0 }  ][fill={rgb, 255:red, 0; green, 0; blue, 0 }  ][line width=0.75]      (0, 0) circle [x radius= 3.35, y radius= 3.35]   ;

\draw (397.74,232.21) node  [font=\small]  {$x_{1} =x_{k} '$};
\draw (392.56,97.76) node  [font=\small]  {$y_{k} '=y_{1} '=T( x_{1})$};
\draw (454.59,245.61) node  [font=\small]  {$x_{2} =x_{k-1} '$};
\draw (435.41,78.16) node  [font=\small]  {$y_{2} =y_{k-1} '=T( x_{2})$};
\draw (498.44,232) node  [font=\small]  {$\cdots $};
\draw (498.27,97.55) node  [font=\small]  {$\cdots $};
\draw (539.29,246.21) node  [font=\small]  {$x_{k-1} =x_{2} '$};
\draw (573.12,77.76) node  [font=\small]  {$y_{k-1} =y_{2} '=T( x_{k-1})$};
\draw (596.15,231.61) node  [font=\small]  {$x_{k} =x_{1} '$};
\draw (622.97,98.16) node  [font=\small]  {$y_{1} =y_{k} =T( x_{k})$};
\draw (54.56,97.76) node  [font=\small]  {$y_{k} '=y_{1} '=T( x_{1})$};
\draw (98.41,78.16) node  [font=\small]  {$y_{2} =y_{k-1} '=T( x_{2})$};
\draw (161.27,97.55) node  [font=\small]  {$\cdots $};
\draw (236.12,77.76) node  [font=\small]  {$y_{k-1} =y_{2} '=T( x_{k-1})$};
\draw (284.97,98.16) node  [font=\small]  {$y_{1} =y_{k} =T( x_{k})$};
\draw (59.74,232) node  [font=\small]  {$x_{1} =x_{k} '$};
\draw (116.59,245.39) node  [font=\small]  {$x_{2} =x_{k-1} '$};
\draw (161.44,231.79) node  [font=\small]  {$\cdots $};
\draw (202.29,246) node  [font=\small]  {$x_{k-1} =x_{2} '$};
\draw (259.15,231.39) node  [font=\small]  {$x_{k} =x_{1} '$};

\end{tikzpicture}

} 
\end{center}
\vspace{-1em}\caption{Schematic representation of the sums~\eqref{eq:firstSum} (left) and~\eqref{eq:secondSum} (right). Each dashed line stands for a term $c(\cdot,\cdot)$.} 
\label{fi:sums}
\end{figure}
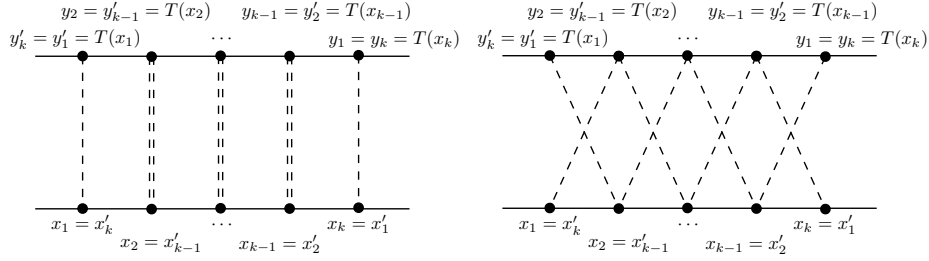
  In summary, we have 
  \begin{align*}
      I(x,y)+ I(x',y') &\geq c(x,y)+ c(x',y')- c(x,y')-c(x',y) + \Xi
  \end{align*} 
  where, always with the conventions $x_{1}=x$ and $x_{k}=x'$,
  \begin{align*}
    \Xi&:=\sup_{k\geq2}\,\sup_{x_{2},\dots,x_{k-1}\in\spt\mu} \Xi_{k} \quad \qforq\\
    \Xi_{k} &:= \sum_{i=1}^{k-1}c(x_{i},T(x_{i}))+ c(x_{i+1},T(x_{i+1}))-c(x_{i},T(x_{i+1}))- c(x_{i+1},T(x_{i})).
  \end{align*} 
 
  It remains to show that given $\eta>0$, we can achieve $\Xi_{k}\geq-\eta$ by a suitable choice of $k$ and $x_{2},\dots,x_{k-1}$.
  Fix a continuous, rectifiable curve $\varphi: [0,1]\to\X_{0}$ with $\varphi(0)=x$ and $\varphi(1)=x'$, and denote its length by $C$. For each $k\geq2$ there exist $0=t_{1}<t_{2}<\dots< t_{k-1}<t_{k}=1$ such that $x_{i}:=\varphi(t_{i})$ satisfy $d(x_{i},x_{i+1})\leq C/(k-1)$ for all $1\leq i\leq k-1$. Applying Assumption~\ref{as:crossContinuity} on the compact set $\varphi([0,1])$, we have that 
  \begin{align}\label{eq:sumI2cont}
  \sum_{i=1}^{k-1} & |c(x_{i},T(x_{i}))+ c(x_{i+1},T(x_{i+1}))-c(x_{i},T(x_{i+1}))- c(x_{i+1},T(x_{i}))| \nonumber\\
  &\leq (k-1) o(C/(k-1)) = o(1)
  \end{align}
  as $k\to \infty$.
\end{proof}

\begin{remark}
The preceding arguments can be generalized to handle certain discontinuities in~$T$, even though at a discontinuity, \eqref{eq:crossContinuity} can only be expected with $o(1)$ rather than $o(d(x_{1},x_{2}))$. Indeed, the conclusion of \eqref{eq:sumI2cont} still holds if for a bounded number of $i$'s, the term under the sum is only $o(1)$. For instance, this can be used to handle the case of semi-discrete transport with quadratic cost,  where~$\nu$ has finite support and hence the transport map is necessarily discontinuous.
\end{remark}

\begin{corollary}\label{co:liminf2}
  Let Assumption~\ref{as:crossContinuity} hold. For any open set $U\subset \X_{0}\times\Y_{0}$,
  $$
    \liminf_{\eps\to0} \eps \log \pi_{\eps}(U) \geq -\inf_{(x,y)\in U} I(x,y).
  $$ 
\end{corollary}  

\begin{proof}
The argument is similar to the proof of Corollary~\ref{co:liminf}, using the inequality~\eqref{eq:Isum2} instead of the equality~\eqref{eq:Isum}. In the course of the argument one also obtains that~\eqref{eq:Isum2} already implies~\eqref{eq:Isum}. We omit the details.
\end{proof}

\section{Positivity of the Rate Function}\label{se:positivity}

The aim of this section is to establish that, under certain conditions, $I(x,y)$ of~\eqref{eq:defI} is strictly positive for $(x,y)\in\X_{0}\times\Y_{0}$ outside the support~$\Gamma$ of the limiting optimal transport~$\pi_{*}$. In view of Corollary~\ref{co:liminf}, this implies that that the mass of~$\pi_{\eps}$ around~$(x,y)$ converges exponentially fast.

When both marginals are supported by finitely many points, it is known that exponential convergence holds for any cost function~\cite{CominettiSanMartin.94, Weed.18}. We shall see that in the continuum case, such a statement must depend on the \emph{geometry} of the cost.
Throughout this section, we assume that $\X=\R^{d}$ (while $\Y$ is Polish). The cost $c$ is continuous and differentiable in~$x$ with $\nabla_{x}c$ continuous, and there exists an optimal transport $\pi_{*}$ as in~\eqref{eq:convOT}. We recall the \emph{twist condition} of optimal transport (e.g., \cite[p.\,234]{Villani.09}) which requires that $\nabla_{x}c(x,\cdot)$ be injective; it holds in particular for the quadratic cost. %
Example~\ref{ex:Ivanishes} below shows that $I$ may vanish on a set of large measure~$\mu\otimes\nu$ when the twist condition does not hold.

Similarly to the preceding section, we present a primal and a dual approach. The direct approach proceeds as follows. Given $(x,y)\notin\Gamma$, we use the geometry of~$c$ and regularity of the optimal transport to find an auxiliary pair $(\tilde x,\tilde y)\in\Gamma$ such that
$c(x,y) - c(\tilde x,y) + c(\tilde x, \tilde y) - c(x, \tilde y)>0$. Then, the definition~\eqref{eq:defI} of~$I$ (with $k=2$) shows that $I(x,y)>0$. The following is one possible implementation.

\begin{lemma}\label{le:positiveTool}
  Fix $(x,y')\in\Gamma$ and $y\in\Y$. Suppose that 
  \begin{equation}\label{eq:defv}
    v:=\nabla_{x}c(x,y) - \nabla_{x}c(x,y')\neq 0
  \end{equation}
  and that there exist $(x_{n},y_{n})\in\Gamma$ such that $(x_{n},y_{n})\to (x,y')$ and 
  \begin{equation}\label{eq:angle}
    \liminf \cos \alpha_{n}>0 \qforq \alpha_{n}:=\angle(v,x-x_{n}).
  \end{equation}
  Then $I(x,y)>0$.
\end{lemma}

\begin{proof}
Set $\Delta(x,y',y_{n}):=\nabla_{x}c(x,y') - \nabla_{x}c(x,y_{n})$; then $\Delta(x,y',y_{n})\to0$ as $d(y',y_{n})\to0$ and we have
\begin{align*}
    \delta_{n}
    :&= c(x,y)  - c(x_{n},y)  + c(x_{n},y_{n}) - c(x, y_{n}) \\  
    &= \br{\nabla_{x}c(x,y) - \nabla_{x}c(x,y_{n}),x-x_{n}} + o(\|x-x_{n}\|) \\
    &= \br{\nabla_{x}c(x,y) - \nabla_{x}c(x,y') + \Delta(x,y',y_{n}),x-x_{n}} + o(\|x-x_{n}\|) \\
    &= \br{v,x-x_{n}} + o(\|x-x_{n}\|) + O(d(y',y_{n}))\|x-x_{n}\| \\
    &= \cos (\alpha_{n}) \|v\| \|x-x_{n}\| + o(\|x-x_{n}\|) + O(d(y',y_{n}))\|x-x_{n}\|.
  \end{align*} 
  As $v\neq 0$ and $\liminf_{n} \cos \alpha_{n}>0$, it follows that $\delta_{n}>0$ for~$n$ large enough. Fix such an~$n$, then choosing $k=2$ and $(x_{2},y_{2}):=(x_{n},y_{n})$ in~\eqref{eq:defI} shows that $I(x,y)\geq \delta_{n}>0$.
\end{proof}

Recall the notation $\X_{0}=\proj_{\X}\Gamma$ and $\Gamma=\spt \pi_{*}$. If~$x$ is interior in $\X_{0}$, we can choose auxiliary points in any direction from~$x$ and Lemma~\ref{le:positiveTool} yields a positivity result for~$I(x,y)$ as follows.

\begin{lemma}\label{le:positiveInteriorCont}
  Let $x\in\Int \X_{0}$ and $y\in\Y$. Let $\pi_{*}$ be given by a transport map~$T$ which is continuous at~$x$. If 
  $
    \nabla_{x}c(x,y) - \nabla_{x}c(x,T(x)) \neq 0,
  $
  then $I(x,y)>0$.
\end{lemma} 

\begin{proof}
  For $n$ large we can uniquely define a point $x_{n}\in \partial B_{1/n}(x)\subset \X_{0}$ by the requirement that  $x-x_{n}$ be parallel to~$v:=\nabla_{x}c(x,y) - \nabla_{x}c(x,T(x))$ (here $\partial B$ denotes the boundary). Then $\cos \alpha_{n}=1$ in the notation of~\eqref{eq:angle} and we conclude using Lemma~\ref{le:positiveTool} with $(x,y')=(x,T(x))$ and $(x_{n},y_{n})=(x_{n},T(x_{n}))$.
\end{proof} 

Sufficient conditions for the continuity (and higher regularity) of the transport map have been studied extensively; see \cite[Section~12]{Villani.09} for an overview of now-classical results and, among others, \cite{CorderoFigalli.19} for recent results including unbounded domains.

The situation is more delicate if~$x$ is a boundary point of~$\X_{0}$ or a point of discontinuity of the transport map, as that restricts the viable choices for approximating sequences. We provide some examples of possible results; for simplicity of exposition, they are stated for the quadratic cost on $\X=\Y=\R^{d}$. The extension of such arguments to a general class of cost functions  is discussed in Appendix~\ref{se:Loeper}.

\begin{lemma}\label{le:positiveAtBoundaryQuadratic}
  Let $c(x,y)=\|x-y\|^{2}$, let $\X_{0}$ be strictly convex\footnote{In the sense that the open segment $(x,x')$ is contained in $\Int\X_{0}$ for distinct $x,x'\in\X_{0}$.} and consider $(x,y)\in (\X_{0}\times \Y_{0})\setminus \Gamma$ with $x\in \partial \X_{0}$. Suppose that $\pi_{*}$ is given by a transport map~$T$ which is continuous on a neighborhood $B_{r}(x)\cap \X_{0}$ for some $r>0$.
Then $I(x,y)>0$.
\end{lemma} 

\begin{proof}
  The main step is to find a point $x''\in \X_{0}$ such that
  \begin{equation}\label{eq:proofPositiveAtBoundaryQuadratic}
    \br{v,x-x''}>0.
  \end{equation}
  Once that is achieved, we may choose a sequence $x_{n}\to x$ in the open segment $(x'',x)$ which is contained in $\Int\X_{0}$ due to strict convexity. As $(x_{n},T(x_{n}))\to (x,T(x))$ by continuity and $\alpha_{n}=\angle(v,x-x'')$ for all~$n$, we conclude by Lemma~\ref{le:positiveTool} with $(x,y'):=(x,T(x))$. %
  
  To find~$x''$ satisfying~\eqref{eq:proofPositiveAtBoundaryQuadratic}, we first fix $x'\in\X_{0}$ such that $(x',y)\in\Gamma$. As~$c$ is quadratic, we have $v=y'-y$ in~\eqref{eq:defv} and the cyclical monotonicity of~$\Gamma$ yields
  $$
    \br{v,x-x'}=\br{y'-y,x-x'}\geq0.
  $$
  If this inequality is strict, we choose $x'':=x'$. Whereas if $\br{v,x-x'}=0$, we consider the mid-point $\bar{x}=(x'-x)/2$ which satisfies $\bar{x}\in\Int\X_{0}$ by strict convexity as well as  $\br{v,x-\bar{x}}=0$. After choosing $\rho>0$ small enough such that $\partial B_{\rho}(\bar{x})\subset \X_{0}$, we can find a point $x''\in\partial B_{\rho}(\bar{x})\subset \X_{0}$ such that $\br{v,x-x''}>0$, completing the proof.
\end{proof}

Next, we illustrate the dual approach in a problem with \emph{discontinuous} optimal transport map. For the remainder of the section, we assume that there exists a Kantorovich potential $\psi$ such that 
\begin{equation}\label{eq:IandPotentialsPos}
      I(x,y)=c(x,y)-\psi^{c}(y)+\psi(x), \quad (x,y)\in \X_{0}\times\Y_{0}.
\end{equation}
As seen in Proposition~\ref{pr:dualIndep}, a sufficient condition is Assumption~\ref{as:uniquePotential} (uniqueness of potentials). If we assume that $\mu\sim\cL^{d}$ on its support, the quadratic cost and the convexity condition in the below results already guarantee that Assumption~\ref{as:uniquePotential} holds; cf.\ Proposition~\ref{pr:dualUniqueness}. The relevance of~\eqref{eq:IandPotentialsPos} is that it yields the representation
\begin{equation}\label{eq:InullAndcSubdiff}
      \{I=0\} \cap (\X_{0}\times\Y_{0})= \partial_{c}\psi \cap (\X_{0}\times\Y_{0}),
\end{equation}
so that our question regarding exponential convergence can be phrased as: 
\begin{center}
does $\Gamma$ fill the entire set $\partial_{c}\psi \cap (\X_{0}\times\Y_{0})$?
\end{center}
The intersection with $\X_{0}\times\Y_{0}$ is crucial to avoid a negative answer in many cases with discontinuous transport (see also the proof of Proposition~\ref{pr:positiveQuadraticSemidiscrete} below). On the other hand, the intersection is justified because the interpretation of~$I$ as rate of convergence is meaningless outside $\spt\pi_{\eps}$.

We first state the following continuation argument similar to Lemma~\ref{le:positiveAtBoundaryQuadratic}.

\begin{lemma}\label{le:positiveAtBoundaryQuadraticDual} 
  Let $c(x,y)=\|x-y\|^{2}$, let $\X_{0}$ be strictly convex and consider $(x,y)\in (\X_{0}\times \Y_{0})\setminus \Gamma$ with $x\in \partial \X_{0}$. Suppose that $I(\tilde{x},y)>0$ for all $\tilde{x}\in \Int \X_{0} \cap B_{r}(x)$, for some $r>0$. Then $I(x,y)>0$.
\end{lemma} 

\begin{proof}
  We may state the proof with the equivalent cost $c(x,y)=-\br{x,y}/2$, so that the notions of $c$-convex analysis and convex analysis coincide. Suppose for contradiction that $I(x,y)=0$. Fix $x'\in\X_{0}$ such that $(x',y)\in\Gamma$ and denote $\phi:=-\psi^{c}$ for $\psi$ as in~\eqref{eq:IandPotentialsPos}, then both $x$ and $x'$ are in the set
  $$
    \{I(\cdot,y)=0\} = \partial_{c} \phi(y) = \partial\phi(y),
  $$
  where $\partial\phi(y)$ denotes the subdifferential of the convex function $\phi$ in the usual sense. The latter set being convex, it must include the whole segment $[x,x']$, meaning that $I(\tilde{x},y)=0$ for all $\tilde{x}\in[x,x']$. The interior of the segment is included in $\Int \X_{0}$ by strict convexity, contradicting the hypothesis.
\end{proof} 

\begin{proposition}[Semidiscrete Transport]\label{pr:positiveQuadraticSemidiscrete} 
  Let $c(x,y)=\|x-y\|^{2}$ on $\X=\Y=\R^{d}$, let $\X_{0}$ be strictly convex, let $\mu\ll\cL^{d}$ and let $\spt\nu$ be at most countable, with no accumulation points. Then $\{I=0\}\cap (\X_{0}\times \Y_{0})=\Gamma$.
\end{proposition} 

\begin{proof}
  Again, we may state the proof with the equivalent cost $c(x,y)=-\br{x,y}/2$. Let $(x,y)\in\X_{0}\times \Y_{0}$. In view of Lemma~\ref{le:positiveAtBoundaryQuadraticDual}, it suffices to treat the case $x\in\Int \X_{0}$. Denote by $\dom\nabla\psi$ the set of points where $\psi$ is differentiable and assume that $I(x,y)=0$; that is, $y\in\partial_{c}\psi(x)=\partial\psi(x)$. The (ordinary) subdifferential $\partial\psi(x)$ equals $\{\nabla\psi(x)\}$ if $x\in\dom\nabla\psi$, whereas in general, it can be described (cf.\ \cite[Theorem~25.6, p.\,246]{Rockafellar.70})
 as the closed convex hull of
 \begin{align}\label{eq:extremePoints}
  S(x) =\Big\{ \lim_{n\to \infty} \nabla\psi(x_n): \, x_{n} \to x,\, x_{n}\in\dom\nabla\psi, \, \lim_{n\to \infty} \nabla\psi(x_n) \mbox{ exists} \Big\}.
\end{align}
  
\emph{Case~1: $x\in\dom\nabla\psi$.}  As $\Gamma\subset\partial\psi$ and $\partial\psi(x)$ is a singleton, it follows that $(x,y)=(x,\nabla\psi(x))\in\Gamma$.

\emph{Case~2: $y\in S(x)$.} Let $x_{n} \to x$ be as in~\eqref{eq:extremePoints}. Recalling that $x\in\Int \X_{0}$, we have $x_{n}\in \X_{0}$  for $n$ large. Thus $(x_{n},\nabla\psi(x_{n}))\in\Gamma$ by Case~1 and closedness entails that the limit $(x,y)$ pertains to~$\Gamma$ as well.

\emph{Case~3: $y\in \partial\psi(x)\setminus S(x)$.} We shall show that this case does not occur. As a first step, we argue that 
\begin{equation}\label{eq:polygon}
  \partial\psi(x) = \conv S(x)
\end{equation}
in the present context (without taking closure). As $x\in\Int \X_{0}\subset\Int \{\psi<\infty\}$, the subdifferential $\partial\psi(x)$ is bounded \cite[Theorem~23.4, p.\,217]{Rockafellar.70}. Let $U$ be a bounded neighborhood of $\partial\psi(x)$. The discreteness assumption on $\spt\nu$ entails that $U\cap\Y_{0}$ is a finite set (and that $\Y_{0}=\spt\nu$).
Let $x_{n} \to x$ be as in~\eqref{eq:extremePoints}. For $x_{n}$ close to $x$ we have $\nabla\psi(x_{n})\in U$, but also $\nabla\psi(x_{n})\in\Y_{0}$ by Case~1. As a result, the set $S(x)$ of limits is finite. In particular, its convex hull is already closed, and~\eqref{eq:polygon} follows.

Now let $y \in \partial\psi(x)\setminus S(x)$. By~\eqref{eq:polygon}, $y$ is a nontrivial convex combination $y=\sum_{i=1}^{k}\theta_{i}y_{i}$ for some distinct $  y_{i}\in S(x)$ and $\theta_{i}\in(0,1)$ with $\sum\theta_{i}=1$. Let $\phi:=-\psi^{c}$ (which is the Legendre--Fenchel transform of~$\psi$ in this context) and $x'\in \partial\phi(y)$. Then cyclical monotonicity of $\partial\phi$ implies $\br{x'-x,y-y_{i}}\geq0$ for all~$i$ and as 
\begin{equation}\label{eq:proofPositiveCyclMon}
  \sum_{i=1}^{k}\theta_i\br{x'- x, y- y_i} = \br{x'- x, 0} =0,
\end{equation}
it follows that $\br{x'- x, y- y_i}=0$ for all~$i$. That is, we have
$$
  \partial\phi(y) - \{x\} \; \perp \; y-y_{i} \qforallq 1\leq i\leq k,
$$
which implies in particular $\dim\partial\phi(y)<d$. On the other hand, $\nu(\{y\})>0$ by the discreteness of~$\Y_{0}$. Thus $\mu(\partial\phi(y))=\nu(\{y\})>0$, contradicting that $\mu\ll\cL^{d}$ does not charge lower dimensional sets. This shows that Case~3 does not occur and completes the proof.
\end{proof} 

The preceding arguments can be extended to a class of cost functions satisfying a Ma-Trudinger-Wang condition. This is detailed in Appendix~\ref{se:Loeper}.

\begin{proposition}\label{pr:positiveLoeper}
  After replacing convexity by $c$-convexity, Lemma~\ref{le:positiveAtBoundaryQuadraticDual} and Proposition~\ref{pr:positiveQuadraticSemidiscrete} extend to cost functions~$c$ satisfying Assumption~\ref{as:Loeper}.
\end{proposition} 

We conclude with a simple example illustrating the relevance of the twist condition. Here, $\nabla_{x}c(x,y)$ vanishes below the diagonal, so that the condition fails, and indeed the convergence $\pi_{\eps}\to\pi_{*}$ is sub-exponential in that region.

\begin{example}[No Twist]\label{ex:Ivanishes}
  Consider $\X=\Y=\R$ with identical marginals $\mu=\nu$ having support $[0,1]$ and the cost function
  $$
  c(x,y)=\begin{cases}
(y-x)^{2}, & y\geq x, \\
0, & y< x.
\end{cases} 
  $$
As $\nabla_{x}c(x,y)=0$ for all $y<x$, this cost does not satisfy the twist condition. Clearly there is a unique optimal transport $\pi_{*}\in\Pi(\mu,\nu)$, given by the Monge map $T(x)=x$. Its support is $\Gamma=\{(x,x): 0\leq x\leq 1\}$  and one can check by direct calculation based on~\eqref{eq:defI} that $I=c$
on $\X_{0}\times\Y_{0}=[0,1]^{2}$. Assumption~\ref{as:crossContinuity} is readily verified, hence Corollary~\ref{co:liminf2} shows that~$I$ is indeed the rate function in this context. We can obtain the same conclusion from Corollary~\ref{co:liminf}, at least if we also suppose that~$\mu$ is equivalent to the Lebesgue measure on $[0,1]$: then, Proposition~\ref{pr:dualUniqueness} shows that Assumption~\ref{as:uniquePotential} holds. Or as a third option, we may verify directly that~$I$ satisfies~\eqref{eq:Isum}, and then conclude as in the proof of Corollary~\ref{co:liminf}. In any event, we see that $I=0$ on~$\{y<x\}$, indicating sub-exponential decay of the weight of $\pi_{\eps}$.
\end{example}

\appendix

\section{Cyclical Invariance and Factorization}\label{se:factorization}

In this section we detail some classical facts about static Schr\"odinger bridges as well as the proof of Proposition~\ref{pr:cyclInv}.
Let $(\X,\mu)$ and $(\Y,\nu)$ be Polish probability spaces; as before, we denote by $\Pi(\mu,\nu)$ the set of couplings.

\begin{proposition}\label{pr:factorization}
Let $R$ be a probability measure on $\X\times\Y$ and suppose that 
\begin{equation}\label{eq:finiteEntropyCoupling}
 \mbox{there exists $\pi\in\Pi(\mu,\nu)$ with $H(\pi|R)<\infty$.}
\end{equation}
Then there is a unique minimizer $\pi^{*}\in\Pi(\mu,\nu)$ for 
$
  \inf_{\pi\in\Pi(\mu,\nu)} H(\pi|R).
$
Assume in addition that~$R\sim \mu\otimes\nu$. Then $\pi_{*}\sim\mu\otimes\nu$ and there exist measurable functions $Z: \X\times\Y\to (0,\infty)$, $f: \X\to (0,\infty)$, $g: \Y\to (0,\infty)$ such that 
\begin{equation}\label{eq:factorization}
  Z(x,y)=f(x)g(y), \quad (x,y)\in\X\times\Y
\end{equation}
and $Z$ is a version of the Radon--Nikodym density $d\pi^{*}/dR$. Conversely, if $\pi\in\Pi(\mu,\nu)$ and a version of its density has the form~\eqref{eq:factorization} on a set of full $\mu\otimes\nu$-measure, where $f$ and $g$ are arbitrary $[-\infty,\infty]$-valued functions, then $\pi=\pi^{*}$.

The uniqueness result also holds without Assumption~\eqref{eq:finiteEntropyCoupling}, if stated as follows. Let $\pi,\pi'\in\Pi(\mu,\nu)$ and $\pi,\pi',R\sim \mu\otimes\nu$. If versions of $d\pi/dR$ and $d\pi'/dR$ both admit factorizations as above, then $\pi=\pi'$.
\end{proposition} 

\begin{proof}
The result under~\eqref{eq:finiteEntropyCoupling} can be found in \cite[Theorem~2.1]{Nutz.20} in the stated form (where we do not assume a priori that one can choose $\pi\sim R$  in~\eqref{eq:finiteEntropyCoupling}).

For the final generalization on the uniqueness claim, let $\pi,\pi'$ be as stated and note that a version of the density $d\pi/d\pi'$ then admits a factorization. We consider $\pi'$ as an auxiliary reference measure, instead of~$R$. Then the analogue of~\eqref{eq:finiteEntropyCoupling} holds as $\pi'$ is itself a coupling and clearly $\pi'$ is the unique minimizer of $H(\cdot|\pi')$. We can now apply the above results.
\end{proof}

We mention that the existence and uniqueness of~$\pi_{*}$ are due to \cite{Csiszar.75}, and that the %
factorization of the density and its measurability are delicate in general (see \cite{BorweinLewis.92, BorweinLewisNussbaum.94, FollmerGantert.97, RuschendorfThomsen.97}, among others) but less so under our condition that $R\sim\mu\otimes\nu$.
An insightful approach with a direct construction of the factorization was recently proposed in \cite{BackhoffBeiglbockConforti.21}; it yields similar results for the entropic function $h(x)=x\log x$ considered here but also allows for a generalization to nonconvex penalties~$h$. In addition, it portrays what we called cyclical invariance as the cyclical monotonicity of an optimal transport problem arising from the linearization of the static Schr\"odinger bridge problem. Another recent work, \cite{BaradatLeonard.20}, uses Markovian methods to obtain a generalized factorization result for Schr\"odinger bridges with additional constraints.
  
\begin{proof}[Proof of Proposition~\ref{pr:cyclInv}]
  Recall the definition~\eqref{eq:defR} of~$R$ and note that $R\sim P=\mu\otimes\nu$. The entropic optimal transport problem~\eqref{eq:EOT} can we rewritten as
  $
    \inf_{\pi\in\Pi(\mu,\nu)} \eps H(\pi|R),
  $
  putting it in the realm of Proposition~\ref{pr:factorization}. Similarly, \eqref{eq:finitenessCond} is equivalent to~\eqref{eq:finiteEntropyCoupling}.
  Let $Z$ be as in~\eqref{eq:factorization}, then~\eqref{eq:cyclInvR} follows, and hence also~\eqref{eq:cyclInv}.
  
  Conversely, if $\pi\in\Pi(\mu,\nu)$ is cyclically invariant, then $\pi\sim P$ and ~\eqref{eq:cyclInvR} holds for its density~$Z$. Fix an arbitrary $x_{0}\in\X$ and note that $f(x):=Z(x,y)/Z(x_{0},y)$ is independent of~$y$ due to~\eqref{eq:cyclInvR} with $k=2$. Setting $g(y)=Z(x_{0},y)/f(x_{0})$ then yields the (measurable) factorization $Z(x,y)=f(x)g(y)$, and we conclude by Proposition~\ref{pr:factorization}. Alternately, the existence of a factorization can be deduced from~\eqref{eq:cyclInvR} by the general result of \cite[Theorem~3.3]{BorweinLewis.92}.
\end{proof} 

\begin{remark}\label{rk:cyclInvarianceFor2}
  The above proof shows that if the cyclical invariance condition~\eqref{eq:cyclInv} holds for $k=2$, then it already holds for arbitrary $k\geq2$.
\end{remark} 

\section{Uniqueness of Potentials}\label{se:uniquenessOfPotentials}

\begin{definition}\label{de:uniquenessOfPotentials}
  Let $\Gamma\subset \X\times\Y$ and $\Lambda\subset \X$. We say that \emph{uniqueness of potentials holds on~$\Lambda$} if for any $c$-convex functions $\psi_{1},\psi_{2}$ on~$\X$ with $\Gamma\subset\partial_{c}\psi_{i}$, it holds that $\psi_{1}-\psi_{2}$ is constant on~$\Lambda$.\end{definition} 

We detail two classes of optimal transport problems where uniqueness of potentials holds. Connectedness of at least one marginal support is essential---uniqueness fails even for the simplest discrete problem, $\mu=\nu=(\delta_{\{1\}}+\delta_{\{2\}})/2$ with cost $c(\{i\},\{j\})=\1_{i\neq j}$.

\begin{proposition}\label{pr:dualUniqueness}
  Let $\X=\R^{d}$ and $\mu\sim\cL^{d}$ on $\spt\mu$, where $\cL^{d}(\partial\spt\mu)=0$ and $\Int\spt\mu$ is connected. Let $\Gamma=\spt\pi$ where $\pi\in\Pi(\mu,\nu)$ is an optimal transport for the continuous cost function $c$.

  \vspace{.3em}\noindent (a) \emph{Lipschitz cost:} Suppose that
  \begin{center}
    $c(\cdot,y)$ is differentiable for all $y$, and locally Lipschitz uniformly in~$y$.
  \end{center}
  Then uniqueness of potentials holds on~$\spt\mu$, and in particular on $\proj_{\X}\Gamma$.
  
  \vspace{.5em}\noindent (b) \emph{Convex, superlinear cost:} Let $\Y=\R^{d}$ and $c(x,y)=h(y-x)$, where 
  \begin{enumerate}
  \item $h: \R^{d}\to\R$ is convex and differentiable,
  \item $h$ has superlinear growth: $h(x)/\|x\|\to\infty$ whenever $\|x\|\to\infty$,
  \item given $r<\infty$ and $\theta\in (0,\pi)$, and for $p\in\R^{d}$ sufficiently far from the origin, there is a cone of the form $\{x\in\R^{d}:\,\|x-p\|\,\|z\|\cos(\theta/2)\leq \br{z,x-p}\leq r\|z\|\}$ for some $z\in \R^{d}\setminus \{0\}$ on which $h$ assumes its maximum at $p$.
  \end{enumerate}  
  Then uniqueness of potentials holds on $\proj_{\X}\Gamma$.
\end{proposition}

\begin{remark}\label{rk: }
  (a) If $c\in C^{1}(\R^{d}\times \R^{d'})$ and $\nu$ is compactly supported, we can always change $c$ outside a neighborhood of $\spt\mu \times\spt\nu$ to satisfy the condition of~(a), without affecting the set of optimal transports.
    
  (b) The convex cost with superlinear growth is essentially the well-known setting of Gangbo and McCann~\cite{GangboMcCann.96}; cf.\ their hypotheses (H2)--(H4). The technical condition~(iii) is implied by (ii) in the radial case $h(x)=\tilde{h}(\|x\|)$; in particular, all the conditions are satisfied for $c(x,y)=\|y-x\|^{p}$ with $p\in (1,\infty)$. In contrast to the main result of \cite{GangboMcCann.96}, $h$ is not assumed to be strictly convex---strictness is required for uniqueness of optimal transports, but not for uniqueness of potentials. For instance, the ``parabola with an affine piece,'' given by $h(x)=\tilde{h}(\|x\|)$ with $\tilde{h}(t)=t^{2}\1_{[0,1]}+ (2t-1)\1_{(1,2)}+(t^{2}-2t+3)\1_{[2,\infty)}$, satisfies all the assumptions in~(b). The affine piece will lead to non-uniqueness of optimal transports for a large class of marginals in the one-dimensional case.
  
  (c) Dual uniqueness may fail if $c$ is not differentiable. For $c(x,y)=|y-x|$ on $\R\times\R$, the $c$-convex functions are exactly the 1-Lipschitz functions. If $\mu=\nu$ is the Lebesgue measure on $[0,1]$, the identical transport $\pi$ is optimal and any 1-Lipschitz function $\psi$ satisfies $\Gamma=\{(x,x):\,x\in [0,1]\}\subset \partial_{c}\psi$. 
\end{remark}

The proof of the proposition is based on the following standard consideration (e.g., \cite[Lemma~3.1]{GangboMcCann.96}). %

\begin{lemma}\label{le:gradientUnique}
  Let $\Gamma\subset \X\times\Y$ and let $\psi,\phi$ be $\overline\R$-valued functions such that $\phi(y)-\psi(x)\leq c(x,y)$ on $\X\times\Y$ and
  $\phi(y)-\psi(x)=c(x,y)$ on $\Gamma$. If $\X=\R^{d}$ and $(x,y)\in\Gamma$ are such that $\psi$ and $c(\cdot,y)$ are differentiable at~$x$, then $\nabla\psi(x)=-\nabla_{x}c(x,y)$.
  In particular, if $c(\cdot,y)$ is differentiable for all $y\in\Y$, then $\nabla\psi(x)$ is uniquely determined for $x\in\proj_{\X}\Gamma\cap\dom\nabla\psi$.
\end{lemma} 

\begin{proof}
  Let $(x,y)\in\Gamma$ be as stated. Then
  \begin{align*}
   \psi(x)+\nabla\psi(x)\cdot h + o(h)
    = \psi(x+h) 
   &\geq \phi(y)-c(x+h,y) \\
   &= \psi(x) +  c(x,y) - c(x+h,y) \\
   &= \psi(x) - \nabla_{x}c(x,y)\cdot h + o(h)
  \end{align*} 
  and hence $\nabla\psi(x)= -\nabla_{x}c(x,y)$ as the direction of $h$ is arbitrary.
\end{proof}

\begin{lemma}\label{le: }
   Let $\Gamma=\spt\pi$ for some $\pi\in\Pi(\mu,\nu)$. Then
   $
     \spt \mu = \overline{\proj_{\X}\Gamma}.
   $
\end{lemma} 

\begin{proof}
Let $(x,y)\in\Gamma$, then $\mu(B_{r}(x))=\pi(B_{r}(x)\times\Y)>0$ for all $r>0$. This shows $\proj_{\X}\Gamma \subset \spt \mu$. Let $x\in \spt \mu$. As $\mu(B_{r}(x))>0$, there must be some $x'\in B_{r}(x)$ with $x'\in \proj_{\X}\Gamma$, and this holds for all $r>0$. Hence, $\spt \mu \subset \overline{\proj_{\X}\Gamma}$.
\end{proof}

\begin{proof}[Proof of Proposition~\ref{pr:dualUniqueness}.] We denote by $\dom \psi$ the set where a function $\psi$ is finite and by $\dom \nabla\psi$ the subset where it is differentiable.

  (a) Let $\psi$ be a $c$-convex function on~$\X=\R^{d}$ with $\Gamma\subset\partial_{c}\psi$. The local Lipschitz bound of $c(\cdot,y)$ implies the same bound for $\psi$. In particular, $\psi$ is continuous and $\cL^{d}$-a.e.\ differentiable  on~$\dom \psi=\R^{d}$.  The coupling property guarantees that $\proj_{\X}\Gamma\subset\spt\mu$ has full $\mu$-measure, hence also full $\cL^{d}$-measure. It follows that  $\Lambda:=\dom\nabla\psi\cap\proj_{\X}\Gamma\subset\spt\mu$ has full $\cL^{d}$-measure, and  $\nabla\psi$ is uniquely determined on~$\Lambda$ by Lemma~\ref{le:gradientUnique}. As $\psi$ is locally Lipschitz and $\Int\spt\mu$ is open and connected, this implies that $\psi$ is uniquely determined (up to constant) on $\Int\spt\mu$ (see, e.g., \cite[Formula~2]{Qi.89}). By continuity on $\R^{d}$, it is also determined on the closure, which equals $\spt\mu$ due to $\cL^{d}(\partial\spt\mu)=0$.

  (b) In this setting, the local Lipschitz property will only hold within $\Int\spt\mu$ and $\psi$ need not be continuous (or even finite) up to the boundary. As we require uniqueness at all (rather than almost all) points~$x$, we argue the boundary case in a second step.
  
  \emph{Step~1.} We first show that uniqueness of potentials holds on $\Int\spt\mu$. It is proved in \cite[Proposition~C.3 and Corollary~C.5]{GangboMcCann.96} that for any $c$-convex function $\psi$ there is a convex set $K$ with $\Int K \subset \dom\psi \subset K$ and that $\psi$ is locally Lipschitz (hence $\cL^{d}$-a.e.\ differentiable) within $\Int\dom\psi$. 
  By convexity, $\Int\overline{K}=\Int K =\Int\dom\psi$. %
  If $\Gamma\subset\partial_{c}\psi$, then $\proj_{\X}\Gamma\subset\dom\psi$ and hence
  $
     \spt \mu = \overline{\proj_{\X}\Gamma} \subset \overline{\dom\psi} \subset \overline{K},
  $
  showing that
  \begin{equation*}\label{eq:intSupportIntDomain}
   \Int \spt\mu \subset \Int\overline{K}=\Int\dom\psi.
  \end{equation*}
  It follows that $\psi$ is locally Lipschitz and $\cL^{d}$-a.e.\ differentiable on $\Int \spt\mu$. On the other hand, $\proj_{\X}\Gamma$ has full $\mu$-measure in $\Int \spt\mu$ by the coupling property, hence also full $\cL^{d}$-measure. Thus $\Lambda:=\dom\nabla\psi\cap \proj_{\X}\Gamma$ has full $\cL^{d}$-measure within $\Int\spt\mu$ and we conclude as in~(a).
  
\emph{Step~2.} 
Define $\X_{1}:=\proj_{\X}\Gamma \cap \Int\spt\mu$. Then
\begin{equation}\label{eq:interiorDenseGamma}
 \Gamma = \overline\Gamma_{1} \qwhereq \Gamma_{1}:=\{(x,y):\, x\in\X_{1},\,y\in\Gamma_{x}\},
\end{equation}
where $\Gamma_{x}$ denotes the section $\{y\in\Y:\,(x,y)\in\Gamma\}$. Indeed, $\mu(\X_{1})=1$ as stated in Step~1, which implies $\pi(\Gamma_{1})=1$ and hence $\Gamma \subset \overline\Gamma_{1}$ by the definition of $\Gamma=\spt\pi$. Conversely, $\Gamma_{1}\subset\Gamma$ is clear, and then $\overline\Gamma_{1}\subset\Gamma$ by closedness.

Fix $(x,y)\in \Gamma$. By~\eqref{eq:interiorDenseGamma} we can find $(x_{n},y_{n})\in\Gamma_{1}$ with $(x_{n},y_{n})\to(x,y)$ and in particular
$$
  \psi^{c}(y)-\psi(x) = c(x,y) = \lim c(x_{n},y_{n}) = \lim \, [\psi^{c}(y_{n})-\psi(x_{n})].
$$
The  $c$-convex functions $-\psi^{c}$ and $\psi$ are lower semicontinuous thanks to the continuity of~$c$, so that $\psi^{c}(y)\geq \limsup \psi^{c}(y_{n})$ and $-\psi(x)\geq \limsup -\psi(x_{n})$. Together, it follows that $\psi^{c}(y)= \lim \psi^{c}(y_{n})$ and $\psi(x)= \lim \psi(x_{n})$. As $x_{n}\in\Int\spt\mu$,  we know from Step~1 that $\psi(x_{n})$ is uniquely determined, and then so is $\psi(x)$.
\end{proof}

\section{Proof of Proposition~\ref{pr:positiveLoeper}}\label{se:Loeper}

In this section, we discuss how to extend Proposition~\ref{pr:positiveQuadraticSemidiscrete} to a general class of cost functions~$c$ satisfying the Ma-Trudinger-Wang condition ``(Aw)'' introduced in~\cite{MaTrudingerWang.05}; we use Loeper's equivalent geometric characterization~\cite{Loeper.09} to generalize from the quadratic case. We recall that the dual representation~\eqref{eq:IandPotentialsPos} of~$I$ has been assumed.

A number of terms from $c$-convex analysis are needed. For ease of reference, we (mostly) follow the notation of~\cite{Loeper.09}, whose Section~2 also provides an excellent introduction to the notions used below. Consider a $C^{1}$ function $c(x,y)$ on the product of two domains $\Omega, \Omega'\subset \R^d$ and suppose that $c$ satisfies the twist condition in both variables; i.e., $\nabla_{x}c(x,\cdot)$ and $\nabla_{y}c(\cdot,y)$ are injective. Given $x\in\Omega$, the $c$-\emph{exponential} map $\fT_{x}$ is defined by $\fT_{x}=-\nabla_x c(x,\cdot)^{-1}$. A $c$-\emph{segment wrt.~$x$} is the image of a segment (in the usual sense) under the map $\fT_{x}$. The $c$-\emph{segment of $y_1,y_2\in \Omega'$ wrt.~$x$} is the image of the segment joining $-\nabla_x c(x,y_1)$ and $-\nabla_x c(x,y_2)$ under~$\fT_{x}$. The set~$\Omega'$ is \emph{$c$-convex wrt.~$\Omega$} if the $c$-segment of $y_1,y_2$ wrt.~$x$  is contained in~$\Omega'$ for all $y_1, y_2\in \Omega'$ and $x\in \Omega$, or equivalently, if $-\nabla_x c(x,\Omega')$ is convex for $x\in \Omega$. \emph{Strict} $c$-convexity means that, in addition, the interior of the $c$-segment is in the interior of~$\Omega'$.  A proper function $\psi:\Omega \to \R\cup \{+\infty\}$ is \emph{$c$-convex} if if there exists $\zeta: \Omega'\to[-\infty,\infty]$ such that $\psi(x)=\sup_{y\in\Omega'} [\zeta(y)-c(x,y)]$. The \emph{$c$-transform} of $\psi$ is defined by $\psi^{c}(y) := \inf_{x\in \Omega} [c(x,y) + \psi(x)]$ for $y\in\Omega'$ and its \emph{$c$-subdifferential} at~$x$ is the set $\partial_{c}\psi(x)=\{y\in\Omega': \, \psi^{c}(y) -\psi(x)=c(x,y)\}$. 

The function $\psi$ is \emph{semiconvex} if it is the sum of a convex function and a function of class $C^{1,1}$. Its (ordinary) subdifferential $\partial \psi(x)$ at $x\in \Omega$ is
$$
\partial \psi(x):=\{y\in\R^{d}: \, \psi(x')\geq \psi(x)+\br{y,x'-x} +o(\|x-x'\|),\; x'\in \Omega \}.
$$ 
Clearly $\partial \psi(x)$ is convex. Moreover, it coincides with the subdifferential of convex analysis if $\psi$ is convex, and it satisfies an analogue of the cyclical monotonicity of convex analysis: adding up the defining inequalities shows
\begin{equation}\label{eq:generalizedCyclMon}
  \br{y-y',x-x'}\geq o(\|x-x'\|) \qforq y\in\partial \psi(x), \quad y'\in\partial \psi(x').
\end{equation}

We shall use analogous notation for functions on~$\Omega'$ instead of~$\Omega$ (a minor abuse of notation since~$c$ is then used with its variables exchanged).

\begin{assumption}\label{as:Loeper}
Let $\Omega,\Omega'$ be domains in $\R^d$ with $\X_0\subset \Omega$ and $\Y_0\subset \Omega'$, and let $c\in C^{1}$ satisfy the twist condition in both variables. Moreover, let $\X_{0}$ be  strictly $c$-convex wrt.~$\Y_{0}$ and let $\Omega'$ be $c$-convex wrt.~$\Omega$. Finally, we assume that any $c$-convex function $\psi$ on $\Omega$ is locally semiconvex and satisfies 
  \begin{equation}\label{eq:cSubdiffAndSubdiff}
  -\nabla_x c(x,\partial_{c} \psi (x)) = \partial \psi(x),
  \end{equation}
  and that the analogue holds for functions on~$\Omega'$.
\end{assumption}

The main condition is~\eqref{eq:cSubdiffAndSubdiff}. As $\partial \psi(x)$ is convex, it implies in particular that $\partial_{c}\psi(x)$ is $c$-convex. (The converse implications also holds; see~\cite{Loeper.09}. Note that our notation differs slightly from~\cite{Loeper.09}, where $\partial_{c} \psi (x)$ denotes what is $-\nabla_x c(x,\partial_{c} \psi (x))$ in our notation.)  It is shown in~\cite{Loeper.09} how~\eqref{eq:cSubdiffAndSubdiff} can be deduced from the~(Aw) condition when the the domains are bounded, sufficiently $c$-convex and $c\in C^{4}$. Local semiconvexity of $c$-convex functions can be ensured by comparably mild conditions on the data, see for instance~\cite[Proposition~2.2]{Loeper.09} or~\cite[Corollary~C.5]{GangboMcCann.96}.  
Apart from the quadratic cost, another classical example treated in~\cite{Loeper.09} is the reflector-antenna cost $c(x,y) = -\log\|x-y\|$. See also~\cite{Villani.09} for further background.

\begin{proof}[Proof of Proposition~\ref{pr:positiveLoeper}.] 
\emph{Step~1: Generalization of Lemma~\ref{le:positiveAtBoundaryQuadraticDual}.} This extension is straightforward: using the same notation as in the proof of Lemma~\ref{le:positiveAtBoundaryQuadraticDual}, we again have $x,x'\in\{I(\cdot,y)=0\} = \partial_{c} (-\psi^{c})(y)$. The latter set is $c$-convex by Assumption~\ref{as:Loeper}, hence contains the $c$-segment of $x,x'$ wrt.~$y$. The interior of the segment is contained in~$\Int\X_{0}$ by strict $c$-convexity, and it includes points from the neighborhood were~$I$ was assumed to be positive---a contradiction.

\emph{Step~2: Generalization of Proposition~\ref{pr:positiveQuadraticSemidiscrete}.}
  Let $(x,y)\in\X_{0}\times \Y_{0}$ be such that $I(x,y)=0$. In view of Step~1, it again suffices to treat the case $x\in\Int \X_{0}$. Moreover, as the $c$-convex function $\psi$ is semiconvex by our assumption, it still holds that $\partial\psi(x)$ is the closed convex hull of $S(x)$ as defined in~\eqref{eq:extremePoints}. The proofs for Case~1 and Case~2 carry over by simply replacing $\partial\psi(x)$ with $\partial_{c}\psi(x)$ and $\nabla\psi(x)$ with $\fT_{x}(\nabla\psi(x))$. In Case~3, the proof of~\eqref{eq:polygon} also carries over using semiconvexity. The arguments around~\eqref{eq:proofPositiveCyclMon} can be adapted as follows: Let $\phi:=-\psi^{c}$ and $x'\in \partial\phi(y)$. Then the cyclical monotonicity property~\eqref{eq:generalizedCyclMon} of $\partial\phi$ implies $\br{x'-x,y-y_{i}}\geq o(\|x'-x\|)$ for all~$i$. In view of~\eqref{eq:proofPositiveCyclMon}, it now follows that $\br{x'- x, y- y_i}=o(\|x'-x\|)$ for all~$i$, but noting that the convex set $\partial\phi(y)$ contains the segment~$[x',x]$, this already implies that $\br{x'- x, y- y_i}=0$ for all~$i$. The remainder of the proof is identical.
\end{proof}

\newcommand{\dummy}[1]{}

\end{document}